\documentclass[11pt,letter paper,twoside,final]{article}
\usepackage[english]{babel}
\usepackage{amsfonts}
\usepackage[utf8]{inputenc}
\usepackage{amssymb,amsmath,amsthm}
\usepackage{enumitem}
\usepackage[pdftex]{graphicx}
\usepackage{csquotes}

\PassOptionsToPackage{dvipsnames,table}{xcolor}
\usepackage{tikz}

\usepackage{float}
\usepackage{verbatim}
\usepackage{booktabs}
\usepackage{array}
\newcolumntype{C}[1]{>{\centering\arraybackslash}m{#1}}
\usepackage{adjustbox}
\usepackage[backref=true, maxcitenames=3, maxbibnames=100, backend=biber]{biblatex}  
\usepackage[pdftex]{hyperref}
\usepackage[inner=2.8cm, outer= 2.8cm, top=2.8cm, bottom=2.8cm]{geometry}

\hypersetup{colorlinks, citecolor=black, filecolor=black, linkcolor=black, urlcolor=black}

\DeclareFieldFormat{pages}{#1} 
\DeclareListFormat{author}{\printnames{author}}
\DeclareFieldFormat[article]{title}{#1, \addcomma} 
\DeclareFieldFormat[inproceedings]{title}{#1} 
\DeclareFieldFormat[unpublished]{title}{#1} 
\DeclareFieldFormat[thesis]{title}{\textit{#1}} 
\DeclareFieldFormat[article]{volume}{\textbf{#1}}
\DeclareFieldFormat{note}{#1}
\DeclareFieldFormat{doi}{#1} 
\DeclareListFormat{publisher}{#1}
\theoremstyle{definition}
\newtheorem{theorem}{Theorem}[section]
\newtheorem{lemma}[theorem]{Lemma}
\newtheorem{corollary}[theorem]{Corollary}
\newtheorem{construction}[theorem]{Construction}

\newtheorem{conjecture}[theorem]{Conjecture}
\newtheorem{agreement}[theorem]{Agreement}
\newtheorem{proposition}[theorem]{Proposition}

\def\arraystretch{1,5}
\newcommand{\Cay}{\mathrm{Cay}}

\newcommand{\Aut}{\mathrm{Aut}}

\newcommand{\zanka}{K_1^o}

\def \ZZ {\mathbb {Z}}

\def \A {\textrm{\textbf{A}}}
\def \GG {\Gamma}

\newcommand{\modulo}[3]{#1 \equiv #2 \: (\textrm{mod }#3)}

\newcommand{\oznaka}[3]{#3 \in \Delta_{#1}(#2)}

\newcommand\blfootnote[1]{%
  \begingroup

  \renewcommand\thefootnote{}\footnote{#1}%

  \addtocounter{footnote}{-1}%

  \endgroup

}

\DeclareBibliographyDriver{article}{%
  \usebibmacro{bibindex}%
  \usebibmacro{begentry}%
  \printnames{author} \newunit\addcomma
  \printfield{title} \newunit\addcomma
  \printfield{journaltitle}  
  \printfield{volume}%
  \iffieldundef{number}{}{\printtext[parens]{\printfield{number}}}\newunit
  \printfield{note} \setunit{\addspace}
  \iffieldundef{year}{\addcomma\printfield{doi}}{\printtext[parens]{\printfield{year}} \newunit\addcomma} 
  \printfield{pages}
  \usebibmacro{finentry}}

\DeclareBibliographyDriver{book}{%
  \usebibmacro{bibindex}%
  \usebibmacro{begentry}%
  \printnames{author} \newunit\addcomma
  \printfield{title} \newunit\addcomma
  \printlist{publisher}%
  \iflistundef{location}{}{\newunit\addcomma\printlist{location}} 
  \printtext[parens]{\printfield{year}} 
  \iffieldundef{pages}{\setunit{\adddot}}{\newunit\addcomma 
  \printfield{pages}}
  \usebibmacro{finentry}}

\DeclareBibliographyDriver{unpublished}{%
  \usebibmacro{begentry}%
  \printnames{author} \newunit\addcomma
  \printfield{title} \newunit\addcomma
  \printfield{note} 
  \usebibmacro{finentry}}

\begin{document}
\begin{center}
\Large{\textbf{On nut graphs with two vertex and three edge orbits}} \\ [+4ex]
Ksenija Rozman{\small$^{a,b,*}$}, Primo\v z \v Sparl{\small$^{a, c, d}$} \\ [+2ex]
{\it \small
$^a$Institute of Mathematics, Physics and Mechanics, Ljubljana, Slovenia\\
$^b$University of Primorska, FAMNIT, Koper, Slovenia\\
$^c$University of Ljubljana, Faculty of Education, Ljubljana, Slovenia\\
$^d$University of Primorska, Institute Andrej Maru\v si\v c, Koper, Slovenia\\
}
\end{center}

\blfootnote{

Email addresses:
ksenija.rozman@pef.uni-lj.si (Ksenija Rozman),
primoz.sparl@pef.uni-lj.si (Primo\v z \v Sparl)

* - corresponding author at: Institute of Mathematics, Physics and Mechanics, Ljubljana, Slovenia
}

\hrule

\begin{abstract}
Nut graphs are graphs whose adjacency matrix is singular with one-dimensional null space spanned by a vector with no zero entries. In a recent paper, Bašić, Fowler and Pisanski proved that the automorphism group of a nut graph has more orbits on the edge set than on the vertex set. They classified all orders for which a vertex-transitive nut graph with precisely two edge orbits exists, and conjectured that a nut graph with two vertex and three edge orbits exists for each non-prime order $n \ge 9$.
Motivated by this conjecture, we introduce a very general construction that provides graphs with the desired symmetry properties, and we determine some sufficient spectral and structural conditions under which they are nut graphs. The construction yields infinite families of examples and confirms the above conjecture for all odd non-prime orders up to~$2\,500$ and for at least $99.8$ percent of all odd non-prime orders up to a million. Finally, we present some additional interesting examples of nut graphs with two vertex and three edge orbits that do not arise from this construction.
\end{abstract}

\hrule

\begin{quotation}


\end{quotation}

\section{Introduction} \label{sec:1}

A nut graph is a graph with singular adjacency matrix and one-dimensional null space spanned by an eigenvector having no zero entries. It is well known and easy to see (see for instance \cite{SG98}, where nut graphs were first described by Sciriha and Gutman) that nut graphs are necessarily connected, have no leafs, and are not bipartite. 

Over the last few years, nut graphs have been studied extensively with the main goal being to classify the orders for which nut graphs with some specific properties exist. For instance, the problem of determining the orders~$n$ for which $d$-regular nut graphs exist, was posed in \cite{GPS23} and studied further in \cite{BDF24, BKS22, FGGPS20}.
Additionally, all the pairs $(n,d)$ for which there exists a $d$-regular circulant nut graph of order~$n$ have been characterized through a series of papers \cite{Damnjanovic24, Damnjanovic23, DS22} (see Section~\ref{sec:2} for the definition of circulants and other notations not defined in the introduction).

Recent research has focused on exploring nut graphs with specific symmetry properties \cite{BFP24}. For instance, it was shown that for each nut graph~$\Gamma$ its automorphism group~$\Aut(\Gamma)$ has more orbits in its action on the edge set than in its action on the vertex set of~$\Gamma$. In particular, edge-transitive nut graphs do not exist \cite[Theorem 2]{BFP24}. Furthermore, it was shown that for each even order $n \ge 8$, there exists a vertex-transitive nut graph with precisely two edge orbits \cite[Theorem 11]{BFP24}, whereas vertex-transitive nut graphs of odd order do not exist (see \cite[Theorem 10]{FGGPS20}). 

With the classification of all possible orders of vertex-transitive nut graphs completed, a natural continuation was to classify possible orders of nut graphs with two vertex orbits. This was done in~\cite{BFP24}. In particular, \cite[Theorem 20]{BFP24} states that there are no nut graphs of prime order with precisely two vertex orbits, while \cite[Theorem 23]{BFP24} states that nut graphs with precisely two vertex orbits exist for all non-prime orders~$n \ge 9$. The examples used to prove \cite[Theorem 23]{BFP24} have different numbers of edge orbits. However, the authors of~\cite{BFP24} conjectured that it suffices to restrict to graphs with three edge orbits. More precisely, they posed the following conjecture.

\begin{conjecture}\cite[Conjecture 24]{BFP24}\label{conjecture}
Let $n \ge 9$ be a non-prime integer. Then there exists a nut graph $\Gamma$ of order~$n$ with two vertex and three edge orbits.
\end{conjecture}

In~\cite{BFP24}, this conjecture was confirmed for orders of particular form, including even numbers, multiples of~$3$ and perfect squares, by exhibiting specific families of nut graphs with the desired degree of symmetry \cite[Proposition 15, Proposition 16, Theorem 23]{BFP24}. Moreover, further research was conducted for the remaining orders up to~$300$, yielding some additional examples but still leaving $20$ orders within this range for which no nut graph with two vertex and three edge orbits was known to exist (see \cite[Question 25]{BFP24}). 
The problem thus remained largely open for odd non-prime orders. More importantly, no general construction of nut graphs of odd order with two vertex and three edge orbits was known at that time. This was the main motivation for our investigations. Before proceeding, let us mention that investigation of nut graphs with a prescribed degree of symmetry has recently been explored in other directions, see for instance~\cite{BD26,BF25}. 

The aim of this paper is to present a general construction that yields nut graphs with two vertex and three edge orbits. The paper is structured as follows. In Section~\ref{sec:2}, we fix the terminology and gather known results that will be used throughout the paper. In Section~\ref{sec:3}, we first determine some properties of nut graphs {of} odd order admitting a subgroup of automorphisms having two vertex and three edge orbits (see Lemma~\ref{lemma:1}). Building on this, we present a general construction (Construction~\ref{cons:merge}) of graphs and show that under some additional condition the resulting graphs have two vertex and three edge orbits (Proposition~\ref{prop:o_vo_e}). We then provide some sufficient spectral and structural conditions under which the graphs arising from this construction are nut graphs (Theorem~\ref{theorem:K}).

In Section~\ref{sec:4}, we use Theorem~\ref{theorem:K} to provide several infinite families of nut graphs with {the} desired symmetry properties (Corollaries~\ref{cor:3n}, \ref{cor:tetra} and~\ref{cor:novi}). The members of these infinite families confirm Conjecture~\ref{conjecture} for all odd non-prime orders up to~$2\,500$. We also performed an extensive computational search to see how ``powerful" our results are when one considers larger orders. {In Section~\ref{sec:5}, we gather some findings from the obtained data. For instance, the above corollaries alone provide nut graphs with two vertex and three edge orbits for at least $99.8$ percent of all odd non-prime orders up to~$1\,000\,000$. Moreover, the obtained data suggests that the proportion of odd non-prime orders, not taken care of by the above corollaries, would remain at approximately $0.11$~percent even if we went further than~$1\,000\,000.$}
Finally, we mention that Construction~\ref{cons:merge} does not yield all possible nut graphs with two vertex and three edge orbits. In the last part of the paper we present three noteworthy examples of this kind, emphasizing their structural diversity and the potential for novel constructions of nut graphs with two vertex and three edge orbits.

\section{Preliminaries} \label{sec:2}
In this section we fix notation, and review some definitions and results that will be used throughout the paper. All graphs are assumed to be finite. Moreover, they will almost always be simple with the only exception that we will at times also allow loops. To indicate that we allow this possibility we will speak of a (multi)graph instead of a graph in such a case. For a graph $\Gamma=(V,E)$ we denote the neighborhood of $v \in V$ by $\Gamma(v)$ and the edge between adjacent vertices $u,v \in V$ by~$uv$. An ordered pair $(u,v)$ of adjacent vertices is called an~{\em arc} (from $u$ to $v$). For a subset $V' \subseteq V$ we denote the induced subgraph of $\Gamma$ on $V'$ by~$\Gamma[V']$.

For an integer~$n$ we denote by $\ZZ_n$ both the ring of residue classes modulo~$n$ and the cyclic group of order~$n$ (where all computations are performed modulo~$n$).  A {\em circulant} is a Cayley graph of a cyclic group. More precisely, for an integer~$n$ and a subset $S \subset \ZZ_n$ with $S=-S$ and $0 \notin S$, the circulant $\Cay(\ZZ_n;S)$ is the graph with vertex set~$\ZZ_n$ in which $i,j$ are adjacent if and only if $j-i \in S$. Note that the {\em cycle} $C_n$ and the {\em complete graph} $K_n$ of order~$n$ are both circulants.

The {\em spectrum} of a graph is the (multi)set of its {\em eigenvalues}, that is the eigenvalues of its adjacency matrix. The graph is {\em singular} whenever $0$ is its eigenvalue. It is well known that the spectrum of the complete graph $K_n$ is $\{ n-1, (-1)^{n-1} \}$ and that the spectrum of the cycle $C_n$, $n \ge 3$, is $\{ 2\cos\frac{2\pi a}{n} \colon a \in \ZZ_n \}$. More generally, the spectrum of a circulant $\Cay(\ZZ_n;S)$ is the set of all sums $\sum_{s \in S}\chi(s)$, where $\chi$ is an irreducible character of~$\ZZ_n$ (see for instance \cite[Lemma 9.2]{Godsil93}).   
For a comprehensive study of graph spectra and singular graphs in particular, see for instance \cite{CDS80}. Let $\Gamma=(V,E)$ be a singular graph and let $\mathbf{x}$ be a corresponding eigenvector (often called a {\em kernel eigenvector}). For each $v \in V$ the entry of~$\mathbf{x}$, corresponding to~$v$, is denoted by~$\mathbf{x}(v)$ and is called the {\em label} of $v$ (with respect to~$\mathbf{x}$). 
Note that for all $v \in V$, the so-called {\em local condition} $\sum_{u \in \Gamma(v)}\mathbf{x}(u)=0,$ must hold for $\mathbf{x}$ to be a kernel eigenvector. In other words, the sum of the labels of all the neighbors of~$v$ must be~$0$.

The automorphism group of a graph~$\Gamma$ is denoted by $\Aut(\Gamma)$.
A graph $\Gamma$ is {\em vertex-/edge-/arc-transitive} if $\Aut(\Gamma)$ acts transitively on the set of vertices/edges/arcs of~$\Gamma$. The number of vertex and edge orbits (of~$\Aut(\Gamma)$) is denoted by $o_v(\Gamma)$ and $o_e(\Gamma)$, respectively. If the graph~$\Gamma$ is clear from the context, we simply denote them by $o_v$ and $o_e$, respectively. In this paper, we focus on graphs with two vertex and three edge orbits, that is on graphs with $(o_v, o_e)=(2,3)$. The following result from~\cite{BFP24} will be useful.

\begin{lemma}\cite[Corollary 22]{BFP24} \label{lemma:odd_order} 
Let $\Gamma$ be a nut graph of odd order~$n$ and let $\mathbf{x}$ be a kernel eigenvector. Then $\mathbf{x}^\alpha= \mathbf{x}$ for all $\alpha \in \Aut(\Gamma)$ and the entries of~$\mathbf{x}$ are constant within a given orbit of~$\Aut(\Gamma)$.
\end{lemma}

So-called orbital graphs, the definition of which we now review, play a central role in our construction of nut graphs with two vertex and three edge orbits.
Let $G$ be a group acting transitively on a set~$X$. Then $G$ has a natural (component-wise) action on the set $X \times X$ of ordered pairs. The orbits of this action are called {\em orbitals}. Since $G$ is transitive, we always have the {\em diagonal orbital} $\mathcal{I} = \{(x,x) \colon x \in X\}$. Any other orbital is {\em non-diagonal}. For any orbital $\Delta$, $\Delta^*=\{(y,x) \colon (x,y) \in \Delta\}$ is also an orbital. An orbital $\Delta$ is called {\em self-paired} whenever $\Delta=\Delta^*$. For an orbital~$\Delta$ and an element~$i$ of the underlying set~$X$, we let $\Delta(i)=\{j \in X \colon (i,j) \in \Delta\}$. The corresponding {\em orbital (di)graph} $\mathcal{O}(\Delta)$ of $\Delta$ is the directed graph with vertex set~$X$ and arc set $\Delta$. Note that $G$ acts transitively on the vertex set and on the arc set of $\mathcal{O}(\Delta)$. We point out that if $\Delta$ is self-paired, then $\mathcal{O}(\Delta)$ can be considered as an undirected graph. We do so throughout the paper. Note also that in this case $G$ acts transitively on the arc set of (the undirected graph) $\mathcal{O}(\Delta)$. We generalize this concept by defining the orbital (di)graph $\mathcal{O}(\cup_{1 \le i \le m}\Delta_i)$ of a union of $m$ orbitals $\Delta_1, \Delta_2, \ldots, \Delta_m$ of a transitive action of~$G$ on~$X$ to be the directed graph with vertex set~$X$ and arc set $\cup_{1 \le i \le m}\Delta_i$. Throughout this paper, we follow the convention that whenever we have a permutation group~$G \le S_n$, we think of the underlying set being~$\ZZ_n$ instead of~$\{1,2, \ldots, n\}$. Therefore, the elements of corresponding orbitals are ordered pairs~$(i,j)$, where $i,j \in \ZZ_n$.

We end this section with a brief summary of some properties of a well known graph product that will play an important role in our examples. The {\em direct product} $\Gamma_1 \times \Gamma_2$ of (multi)graphs $\Gamma_1=(V_1,E_1)$ and $\Gamma_2=(V_2,E_2)$ is the (multi)graph with vertex set $V_1 \times V_2$ and edge set $\{ (v_1,u_1)(v_2,u_2) \colon v_1v_2 \in E_1, u_1u_2 \in E_2 \}$. It is well known that  $\Gamma_1 \times \Gamma_2$ is connected if and only if $\Gamma_1$ and $\Gamma_2$ are both connected and at most one of them is bipartite. It is also well known (see \cite{CDS80}) that the adjacency matrix of $\Gamma_1 \times \Gamma_2$ is the Kronecker product of the corresponding adjacency matrices. Moreover, if $Spec(\Gamma_1)$ and $Spec(\Gamma_2)$ denote the spectrum of $\Gamma_1$ and $\Gamma_2$, respectively, the spectrum of $\Gamma_1 \times \Gamma_2$ is the (multi)set of all products $\lambda\mu$, where $\lambda \in Spec(\Gamma_1)$ and $ \mu \in Spec(\Gamma_2)$.
For ease of reference, we state the following two well-known properties of direct products of (multi)graphs (see \cite{HIK11}), which will be very useful for us, as a proposition.

\begin{proposition}\label{prop:direct_product}
Let $\Gamma_1$ and $\Gamma_2$ be (multi)graphs and let $\Gamma = \Gamma_1 \times \Gamma_2$. 
\begin{itemize}
    \item If $\Gamma_1$ and $\Gamma_2$ are both vertex-/arc-transitive, then $\Gamma$ is also vertex-/arc-transitive.
    \item $\Gamma$ is non-singular if and only if both $\Gamma_1$ and $\Gamma_2$ are non-singular.
\end{itemize}
\end{proposition}

\section{A general construction}\label{sec:3}

In this section we present a general construction of graphs of odd order with at most two vertex and three edge orbits, which, under some additional conditions, result in nut graphs with $(o_v,o_e)=(2,3)$.

We begin by determining some basic properties of a nut graph $\Gamma$ of odd order admitting $G \le \Aut(\Gamma)$ with two vertex and three edge orbits. Let $V_1$ and $V_2$ be the two vertex orbits of~$G$ and denote $\Gamma_i=\Gamma[V_i]$ and $n_i=|V_i|$ for each $i \in \{1,2\}$. Note that $G$ acts vertex-transitively on each of $\Gamma_1$ and $\Gamma_2$. For each $i \in \{1,2\}$ and $v \in V_i$, denote the number of neighbors of~$v$ within~$V_i$ by~$k_i$ (i.e., $k_i$ is the valence of~$\Gamma_i$) and the number of neighbors of~$v$ outside~$V_i$ by~$d_i$.

Let $\mathbf{x}$ be a kernel eigenvector of~$\Gamma$. Since $\Gamma$ is of odd order,  Lemma~\ref{lemma:odd_order} implies that $\mathbf{x}^\alpha=\mathbf{x}$ for all $\alpha \in G$ and that the entries of $\mathbf{x}$ are constant within $V_1$ and within $V_2$. One can then normalize $\mathbf{x}$ so that the entries corresponding to~$V_1$ are equal to~$1$ and hence the entries corresponding to~$V_2$ are equal to some~$\ell \ne 0$. From the local condition at any $v \in V_1$ we obtain $k_1+d_1\ell=0$, while the local condition at any $v \in V_2$ yields $k_2\ell+d_2=0$. Since $\Gamma$ is connected, $d_1, d_2 \ge 1$, forcing $k_1, k_2 \ge 1$. It is now clear that the three orbits of the action of~$G$ on~$E(\Gamma)$ are $E(\Gamma_1)$, $E(\Gamma_2)$ and the set of the edges between~$\Gamma_1$ and~$\Gamma_2$. Therefore, $G$ acts edge-transitively on each of~$\Gamma_1$ and~$\Gamma_2$ (note however, that $\Gamma_1$ and~$\Gamma_2$ need not be connected). Moreover, $\ell=-{k_1}/{d_1}=-{d_2}/{k_2}$, forcing $k_1k_2=d_1d_2$. 
Observe also that $n_1d_1=n_2d_2$ (simply double count the edges between $\Gamma_1$ and $\Gamma_2$). Furthermore, if $\gcd(n_1, n_2)=1$, then since $n_1d_1=n_2d_2$ and $d_1\le n_2$, $d_2 \le n_1$, it follows that  $d_1=n_2$ and $d_2=n_1$. Then $n_1n_2=d_2d_1=k_2k_1<n_2n_1$, a contradiction. Hence, $\gcd(n_1,n_2) > 1$. 

Before stating the result that summarizes the findings of the previous paragraph, we introduce notation and terminology regarding the above partition $(V_1, V_2)$ of the vertex set of a graph that will be used throughout the paper. Let $\Gamma$ be a simple graph. We say that $\Gamma$ admits a {\em bi-regular bi-decomposition} $(V_1,V_2)$, if its vertex set $V(\Gamma)$ is the disjoint union $V_1 \cup V_2$ and there exist nonnegative integers $k_1$, $k_2$, $d_1$, $d_2$ such that

\begin{itemize} \setlength{\itemsep}{0pt}
\item for each $i \in \{1,2\}$, the induced subgraph $\Gamma_i=\Gamma[V_i]$ is $k_i$-regular,
\noindent
\item for distinct $i,i' \in \{1,2\}$, each vertex in $V_i$ has exactly $d_i$ neighbors in $V_{i'}$.
\end{itemize}

\noindent 
The $6$-tuple $\langle n_1,k_1,d_1,n_2,k_2,d_2 \rangle$, where $n_i=|V_i|$ for each $i \in \{1,2\}$, is called the {\em parameter $6$-tuple} corresponding to this bi-regular bi-decomposition. Whenever the decomposition $(V_1, V_2)$ will be clear from the context, we will simply speak of the parameter $6$-tuple of the graph in question.

\begin{lemma} \label{lemma:1}
Let $\Gamma$ be a nut graph of odd order admitting $G \le \Aut(\Gamma)$ that has two vertex and three edge orbits and let $V_1$ and $V_2$ be the two vertex orbits of~$G$. Then $(V_1, V_2)$ is a bi-regular bi-decomposition of~$\Gamma$. Let $\langle n_1,k_1,d_1,n_2,k_2,d_2 \rangle$ be the corresponding parameter $6$-tuple. Set $\Gamma_1=\Gamma[V_1]$ and $\Gamma_2=\Gamma[V_2]$, and let $\mathbf{x}$ be a kernel eigenvector of~$\Gamma$. Then the following hold. 
\begin{enumerate}[label=(\alph*), ref=\thelemma(\alph*)]
\item \label{item:eigenvector} For each $i \in \{1,2\}$, the entries of $\mathbf{x}$ corresponding to the vertices of $\Gamma_i$ are all equal. Moreover, we can assume that $\mathbf{x}(v)=1$ for all $v \in V_1$, and $\mathbf{x}(v)=-k_1/d_1$ for all $v \in V_2$.
\item \label{item:valence_condition} $k_1k_2 = d_1d_2$, and $k_1, k_2, d_1, d_2 \ge 1,$
\item \label{item:gcd_condition} $\gcd(n_1, n_2) \ne 1$.
\end{enumerate}
\end{lemma}

Suppose that $\Gamma$ is a graph admitting a bi-regular bi-decomposition with the corresponding parameter $6$-tuple $\langle n_1,k_1,d_1,n_2,k_2,d_2 \rangle$. We say that $\Gamma$ satisfies the {\em valence condition} (with respect to this decomposition) if $k_1k_2=d_1d_2$. 

In the remainder of this section, the objective is to present a general construction yielding graphs with $(o_v,o_e) = (2,3)$ and with a potential of being nut graphs. We first prove (see Lemma~\ref{lemma:singular}) that all graphs admitting a bi-regular bi-decomposition and satisfying the valence condition are singular. Then, using a group theoretic approach, we give a rather general construction (see Construction~\ref{cons:merge}), which is shown to yield graphs admitting a bi-regular bi-decomposition and a subgroup of the automorphism group with two vertex and three edge orbits (see Proposition \ref{prop:o_vo_e}). 
What remains is to determine the conditions under which the null spaces of these graphs are one-dimensional and are spanned by a vector with no zero entries. At the moment, it seems that obtaining such a necessary and sufficient condition for the construction in its full generality is too difficult. This is why in this paper we focus on a specific case of the construction (see Theorem~\ref{theorem:K} and Figure~\ref{pic:construction}). In the next section, we use it to find infinite families of nut graphs with $(o_v,o_e)=(2,3)$.

To avoid having to deal with multiple subscripts we make an agreement that we henceforth denote the two subsets of a bi-regular bi-decomposition by $V$ and $U$ instead of $V_1$ and $V_2$, respectively, (and denote the vertices from $V$ by~$v$ and those from $U$ by~$u$), while still denoting the corresponding parameter $6$-tuple by $\langle n_1,k_1,d_1,n_2,k_2,d_2 \rangle$.

\begin{lemma}\label{lemma:singular}
Let $\Gamma$ be a graph admitting a bi-regular bi-decomposition~$(V,U)$. If $\Gamma$ satisfies the valence condition with respect to~$(V,U)$, then it is singular.
\end{lemma}

\begin{proof}
Let $\langle n_1, k_1, d_1, n_2, k_2, d_2 \rangle$ be the parameter $6$-tuple corresponding to~$(V,U)$. Assign label~$1$ to all the vertices of~$V$ and label~$-k_1/d_1=-d_2/k_2$ to all the vertices of~$U$. Then the local conditions at each $v \in V$ and at each $u \in U$ are clearly met. Therefore, the vector corresponding to this labeling is an eigenvector for the eigenvalue~$0$, implying that $\Gamma$ is singular.
\end{proof}

\begin{construction}  \label{cons:merge}
Let $n > 1$ be an integer and let $G \le S_n$ be a transitive permutation group. Choose two self-paired orbitals $\Delta_1$ and $\Delta_2$ of~$G$, and any orbital $\Delta_3$ of~$G$. Let $m$ and $t$ be positive integers, and let $H \le S_m$ and $K \le S_t$ be transitive permutation groups. Choose a self-paired orbital $\Delta_4$ of~$H$ such that $\Delta_4\ne \mathcal{I}$ if $\Delta_1=\mathcal{I}$, and a self-paired orbital $\Delta_5$ of~$K$ such that $\Delta_5\ne \mathcal{I}$ if $\Delta_2=\mathcal{I}$. Then the graph $\Gamma=\GG(\Delta_1, \Delta_2, \Delta_3, \Delta_4, \Delta_5)$ is the graph of order $n(m+t)$ with vertex set $ V(\Gamma) = V \cup U$, where
\[V = \{ v_i^j \colon i \in \ZZ_n, j \in \ZZ_m\} \quad {\rm and } \quad U = \{ u_i^j \colon i \in \ZZ_n, j \in \ZZ_t\}, \]
 and edge set $E(\Gamma) = E_1 \cup E_2 \cup E_3,$ where
\[ E_1 = \{v_i^jv_{i'}^{j'} \colon (i,i') \in \Delta_1, \; (j,j') \in \Delta_4 \}, \quad
E_2 = \{u_i^ju_{i'}^{j'} \colon (i,i') \in \Delta_2, \; (j,j') \in \Delta_5 \}
\] \[ {\rm and} \quad E_3 = \{ v_i^j u_{i'}^{j'} \colon (i,i') \in \Delta_3, \;  j \in \ZZ_m, j' \in \ZZ_t \}.\] 
\end{construction}

Some remarks regarding Construction~\ref{cons:merge} are in order. 
Clearly, $\Gamma[V]=(V,E_1)$, and so $\Gamma[V]=\mathcal{O}(\Delta_1) \times \mathcal{O}(\Delta_4)$. Similarly, $\Gamma[U]=(U,E_2)$ and $\Gamma[U]=\mathcal{O}(\Delta_2) \times \mathcal{O}(\Delta_5)$. 
Therefore, since at least one of $\Delta_1$ and $\Delta_4$ is non-diagonal and at least one of $\Delta_2$ and $\Delta_5$ is non-diagonal, $\Gamma$ has no loops and is thus simple. Moreover, since for each $i \in \{1,2,4,5\}$, $\Delta_i$ is self-paired, $\Gamma$ is an undirected graph, and $\Gamma[V]$ and $\Gamma[U]$ are both arc-transitive. 
In fact, the group $G \times H \times K$ has a natural action on $\Gamma$ if for each $g \in G$, $h \in H$ and $k \in K$ we define 
\[
(g,h,k)(v_i^j)=v_{g(i)}^{h(j)} \quad \textrm{~and~} \quad (g,h,k)(u_i^j)=u_{g(i)}^{k(j)}.
\]
Moreover, since for each $i \in \{1,2,3\}$, $\Delta_i$ is an orbital of $G$, and $\Delta_4$, $\Delta_5$ are orbitals of $H$ and $K$, respectively, it is clear that each $(g,h,k)$ is an automorphism of $\Gamma$. Consequently, $G \times H \times K \le \Aut(\Gamma)$. Since $G$, $H$ and $K$ are all transitive, $G \times H \times K$ has two vertex orbits (namely $V$ and $U$) and three edge orbits ($E_1$, $E_2$ and $E_3$).
Furthermore, since the edges of $E_1$ have both end-vertices in~$V$, the edges of $E_2$ have both end-vertices in~$U$ and the edges of $E_3$ have one end-vertex in $V$ and one end-vertex in $U$, we thus have the following result.

\begin{proposition} \label{prop:o_vo_e} 
Let $\Gamma=\GG(\Delta_1, \Delta_2, \Delta_3, \Delta_4, \Delta_5)$ and $G$, $H$ and $K$ be as in Construction~\ref{cons:merge}. Then $\Gamma$ admits $G \times H \times K$ as a subgroup of the automorphism group having two vertex orbits and three edge orbits. Moreover, $\Gamma$ is not a graph with $(o_v, o_e) = (2,3)$ if and only if there exists an automorphism of~$\Gamma$, mapping some vertex from $V$ to some vertex from $U$.
\end{proposition}

An immediate corollary of this proposition is that $(V,U)$ is a bi-regular bi-decomposition of the graph $\GG(\Delta_1, \Delta_2, \Delta_3, \Delta_4, \Delta_5)$. Bearing this in mind, we establish notation conventions that will be used throughout the rest of the paper.

\begin{agreement}\label{notation}
Whenever we have a graph $\Gamma = \GG(\Delta_1, \Delta_2, \Delta_3, \Delta_4, \Delta_5)$ from Construction~\ref{cons:merge}, we will always consider the corresponding bi-regular bi-decomposition $(V,U)$ and set $\Gamma_1=\Gamma[V]$ and $\Gamma_2=\Gamma[U]$. Moreover, we refer to the corresponding parameter $6$-tuple $\langle n_1, k_1, d_1, n_2, k_2, d_2 \rangle$ as the parameter $6$-tuple of~$\Gamma$. 
Additionally, for each $i \in \{1,2,3\}$ we let $\kappa_i = |\{j \in \ZZ_n \colon (0,j) \in \Delta_i\}|$. Similarly, we let $\kappa_4 = |\{j \in \ZZ_m \colon (0,j) \in \Delta_4\}|$ and $\kappa_5 = |\{j \in \ZZ_t \colon (0,j) \in \Delta_5\}|$.
\end{agreement}

With reference to Agreement~\ref{notation}, note that $n_1=mn$ and $n_2=nt$. Moreover, for each $i \in \{1,2,4,5\}$, $\kappa_i$ is the valence of $\mathcal{O}(\Delta_i)$, and so $k_1=\kappa_1\kappa_4$ and $k_2=\kappa_2\kappa_5$.

Construction~\ref{cons:merge} is a very general construction, yielding graphs with at most two vertex and three edge orbits. However, in what follows we only consider the special case in which $\Delta_3=\mathcal{I}$. Thus, we consider graphs of the form $\Gamma=\GG(\Delta_1, \Delta_2, \mathcal{I}, \Delta_4, \Delta_5)$ (see Figure~\ref{pic:construction} for a schematic representation). In this case, with reference to Agreement~\ref{notation}, $d_1=t$ and $d_2=m$, implying that the parameter $6$-tuple of $\Gamma$ can be expressed as $\langle nm, k_1, t, nt, k_2, m \rangle$. Consequently, $\Gamma$ satisfies the valence condition whenever $k_1k_2=mt$. Recall that in this case,  Lemma~\ref{lemma:singular} implies that $\Gamma$ is singular.
As already mentioned, our focus is on graphs of odd order, and therefore $n \ge 3$ must be an odd integer, and $m$ and $t$ must be of different parity (recall that $\Gamma$ is of order $n(m+t)$).
It is easy to see that $\Gamma$ is connected if and only if $\mathcal{O}(\Delta_1 \cup \Delta_2)$ is connected.  In particular, the connectedness of $\Gamma$ does not depend on the connectedness of the orbital graphs $\mathcal{O}(\Delta_4)$ and $\mathcal{O}(\Delta_5)$ ($\Delta_4$ and $\Delta_5$ can even both be diagonal).

\begin{figure}[h!] 
\centering
\begin{tikzpicture}[scale=1, every node/.style={circle, fill=white, inner sep=0pt}]
	\def\n{5}
	\def\m{3}
	\def\t{4}
	\def\dist{5} 

    \foreach \i in {0,...,\numexpr\n-3} {
        \foreach \j in {0,...,\numexpr\m-2} {
            \node[circle, fill=black, inner sep=2.0pt] (v\i\j) at (\j, -\i-1) {};
	      \node[anchor=west, yshift=-7pt] at (v\i\j){$v_{\i}^{\j}$};
        }
    }
    \foreach \i in {0,...,\numexpr\n-3} { 
        \node[circle, fill=black, inner sep=2.0pt] (v\i\m-1) at (\m, -\i-1) {};
        \node[anchor=west, yshift=-7pt, xshift=1pt] at (v\i\m-1){$v_{\i}^{m-1}$};
    }
    \foreach \i in {1,...,\numexpr\n-2,\numexpr\n} {
        \node[fill=none] at ({\m-1}, -\i) {\dots};
    }
    
    \foreach \j in {0,...,\numexpr\m-2} { 
        \node[circle, fill=black, inner sep=2.0pt] (v\n-1\j) at (\j, -\n) {};
        \node[anchor=west, yshift=-7pt, xshift=1pt] at (v\n-1\j){$v_{n-1}^{\j}$};
    }
    \foreach \j in {0,...,\numexpr\m-2,\numexpr\m} {
        \node[fill=none] at (\j, -\n+1) {$\vdots$};
    }

\node[fill=none] at (\m-1,-\n+1) {$\ddots$};
\node[circle, fill=black, inner sep=2.0pt] (v\n-1\m-1) at (\m, -\n) {};
\node[anchor=west, yshift=-7pt, xshift=1pt] at (v\n-1\m-1){$v_{n-1}^{m-1}$};

    \foreach \i in {1,...,\numexpr\n-2} { 
        \draw[rounded corners, thick, blue!60, dash pattern=on 5pt off 3pt, line width = 0.7pt] 
            (-0.3, -\i+0.3) rectangle (\m+1.1, -\i-0.6); 
        \node[anchor=east,text=blue!100] at (-0.5, -\i-0.15) {}; 
    }
    \draw[rounded corners, thick, blue!60, dash pattern=on 5pt off 3pt, line width = 0.7pt] 
            (-0.3, -\n+0.3) rectangle ( \m+1.1, -\n-0.6);
    \node[anchor=east,text=blue!100] at (-0.5, -\n-0.15) {};

    \foreach \i in {0,...,\numexpr\n-3} {
        \foreach \j in {0,...,\numexpr\t-2} {
            \node[circle, fill=black, inner sep=2.0pt] (v\i\j) at (\j+\m+\dist, -\i-1) {};
	      \node[anchor=west, yshift=-7pt] at (v\i\j){$u_{\i}^{\j}$};
        }
    }
    \foreach \i in {0,...,\numexpr\n-3} {
        \node[circle, fill=black, inner sep=2.0pt] (v\i\t-1) at (\t+\m+\dist, -\i-1) {};
        \node[anchor=west, yshift=-7pt, xshift=1pt] at (v\i\t-1){$u_{\i}^{t-1}$};
    }
    \foreach \i in {1,...,\numexpr\n-2,\numexpr\n} {
        \node[fill=none] at ({\t+\m+\dist-1}, -\i) {$\ldots$};
    }

    \foreach \j in {0,...,\numexpr\t-2} {
        \node[circle, fill=black, inner sep=2.0pt] (v\n-1\j) at (\j+\m+\dist, -\n) {};
        \node[anchor=west, yshift=-7pt, xshift=1pt] at (v\n-1\j){$u_{n-1}^{\j}$};
    }
    \foreach \j in {0,...,\numexpr\t-2,\numexpr\t} {
        \node[fill=none] at (\j+\m+\dist, -\n+1) {$\vdots$};
    }

\node[fill=none] at (\t+\m+\dist-1,-\n+1) {$\ddots$};
\node[circle, fill=black, inner sep=2.0pt] (v\n-1\m-1) at (\t+\m+\dist, -\n) {};
\node[anchor=west, yshift=-7pt, xshift=1pt] at (v\n-1\m-1){$u_{n-1}^{t-1}$};
    
    \foreach \i in {1,...,\numexpr\n-2} { 
        \draw[rounded corners, thick, magenta!60, dash pattern=on 5pt off 3pt, line width = 0.7pt] 
            (\m+\dist-0.3, -\i+0.3) rectangle (\m+\dist+\t+1.1, -\i-0.6);
        \node[anchor=west,text=magenta!100] at (\m+\dist+\t+1.4, -\i-0.15) {};
    }
    \draw[rounded corners, thick, magenta!60, dash pattern=on 5pt off 3pt, line width = 0.7pt] 
            (\m+\dist-0.3, -\n+0.3) rectangle ( \m+\dist+\t+1.1, -\n-0.6);
    \node[anchor=west,text=magenta!100] at (\m+\dist+\t+1.4, -\n-0.15) {};

    \node[fill=none] at ({\m-1}, -0.4) {$\Delta_4$};
    \node[fill=none] at (-0.7, -\n/2-0.5) {$\Delta_1$};
    \node[fill=none] at ({\m+\dist+\t/2+0.5}, -0.4) {$\Delta_5$};
    \node[fill=none] at (2*\m+\dist+\t-1.5, -\n/2-0.5) {$\Delta_2$};

\draw[rounded corners=35, thick, black!50, dash pattern=on 7pt off 3pt, line width = 0.5pt] 
        (-1.1, 0) rectangle ( \m+1.7, -\n-1.0);
\node[fill=none, inner sep=0pt, text width=1cm] at (0.6,-\n-1.7)
{\smash {$\Gamma_1=\mathcal{O}(\Delta_1)\times\mathcal{O}(\Delta_4)$}};

\draw[rounded corners=35, thick, black!50, dash pattern=on 7pt off 3pt, line width = 0.5pt] 
        (\m+\dist-0.9, 0) rectangle ( \m+\dist+\t+2, -\n-1.0);
\node[fill=none, inner sep=0pt, text width=1cm] at (\m+\dist+\t/2-0.8, -\n-1.7) 
{\smash {$\Gamma_2=\mathcal{O}(\Delta_2)\times\mathcal{O}(\Delta_5)$}};

    \foreach \i in {1,...,\numexpr\n-2,\n} { %
        \node[fill=none] (v\i\m) at (\m+1.08,-\i-0.13) {};
        \node[fill=none] (v\i\m+\dist) at (\m+\dist-0.28,-\i-0.13) {};
            \draw[ultra thick, black!50] (v\i\m) -- (v\i\m+\dist);
        \node[fill=none] (v\i\m) at (\m+1.08,-\i-0.18) {};
        \node[fill=none] (v\i\m+\dist) at (\m+\dist-0.28,-\i-0.18) {};
            \draw[ultra thick, black!50] (v\i\m) -- (v\i\m+\dist);
    }
    \node[fill=none] at (\m+\dist/2+0.4, -0.8){$\Delta_3=\mathcal{I}$};
\end{tikzpicture}
\vspace{-0.5cm}
\caption{A schematic representation of the graph $\GG(\Delta_1, \Delta_2, \mathcal{I}, \Delta_4, \Delta_5)$ from Construction~\ref{cons:merge}. The thick gray edges connecting the blue dashed rectangles to the corresponding magenta dashed rectangles, are to be understood in the sense that for all $i \in \ZZ_n$, $j_1 \in \ZZ_m$ and $j_2 \in \ZZ_t$, the vertices $v_i^{j_1}$ and $u_i^{j_2}$ are adjacent.}
\label{pic:construction} 
\end{figure}
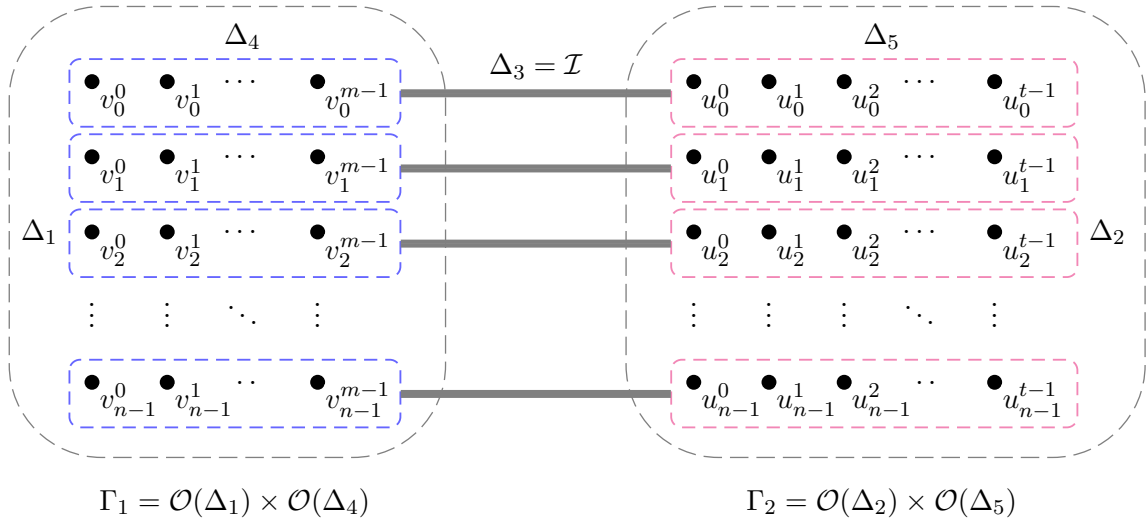

The following lemma will be very useful in our proof of Theorem~\ref{theorem:K}. Before stating it, we review some further well-known notions that will be used from now on. Let $A$ be a square nonnegative matrix, i.e., a square matrix with no negative entries. Then $A$ is said to be {\em irreducible} if for every $i, j$ there exists an integer $r > 0$ such that $(A^r)_{ij}$ is positive. The \textit{underlying digraph} of $A$ is a graph in which there is an arc connecting $i$ to $j$ if $A_{ij} \neq 0$. Clearly, $A$ is irreducible if and only if its underlying digraph is strongly connected. Suppose that the sum of entries in each row of~$A$ is equal to~$\lambda$. It is then easy to see that $\lambda$ is an eigenvalue of~$A$ and that it has the largest absolute value among all eigenvalues of~$A$ (that is, $\lambda$ is the {{\em spectral radius}} of~$A$). Moreover, one of the most fundamental results on the eigenvalues and eigenvectors of nonnegative matrices, known as the Perron-Frobenius theorem, states that if $A$ is {irreducible,} then $\lambda$ has multiplicity~$1$ (for the complete theorem and its proof, see \cite[Chapter 8.3]{BC08} or \cite[Subchapter 8.8]{GR01}). Before finally stating the lemma, we make an agreement that for a digraph~$\Gamma$ its adjacency matrix is denoted by$~\mathbf{\A}(\Gamma)$.

\begin{lemma} \label{lemma:matrix_product}
Let $\Delta_1$ and $\Delta_2$ be two self-paired orbitals of a transitive permutation group~$G$ of degree $n \ge 3$, and let $\Gamma=\mathcal{O}(\Delta_1 \cup \Delta_2)$ be the orbital graph of their union. Set $M_1=\A(\mathcal{O}(\Delta_1))$ and $M_2=\A(\mathcal{O}(\Delta_2))$. Then $M_1M_2$ is irreducible if and only if $\Gamma$ is connected and has an odd cycle.
\end{lemma}

\begin{proof}
Recall that since $\Delta_1$ and $\Delta_2$ are self-paired, $\Gamma$ is an undirected graph. Let $\Lambda$ denote the underlying digraph of $M_1M_2$ with vertices denoted by the elements of $\ZZ_n$. Note that there is an arc from $i$ to $j$ in~$\Lambda$ if and only if there exists a $k \in \ZZ_n$ such that $(i,k) \in \Delta_1$ and $(k,j) \in \Delta_2$. 

Assume that $M_1M_2$ is irreducible, i.e., $\Lambda$ is strongly connected. Then $\Gamma$ is clearly connected. Moreover, for any two adjacent vertices $i,j$ of~$\Gamma$, there exists a walk from $i$ to $j$ in~$\Lambda$. We thus have a walk of even length from $i$ to $j$ in $\Gamma$. Together with the edge $ij$ of~$\Gamma$ we obtain a closed walk of odd length, implying that $\Gamma$ has an odd cycle.

For the other direction, observe that $G\le \Aut(\Lambda)$ and recall that $G$ is transitive. It is well known and easy to see that in this case $\Lambda$ is strongly connected if and only if the underlying undirected graph of~$\Lambda$ is connected. Suppose to the contrary, that the underlying graph of~$\Lambda$ is not connected and denote its connected components by $\mathcal{C}_1$, $\mathcal{C}_2$, \ldots, $\mathcal{C}_s$. By connectedness of~$\Gamma$, there exists an edge $ij$ of~$\Gamma$, connecting two connected components of~$\Lambda$. Without loss of generality, assume that $i \in \mathcal{C}_1$ and $j \in \mathcal{C}_2$. We consider the case that $(i,j) \in \Delta_1$ (the case that $(i,j) \in \Delta_2$ is analogous).
For each $k \in \ZZ_n$ such that $(j,k) \in \Delta_2$, $(i,k)$ is an arc of~$(\Lambda)$, and therefore $k \in \mathcal{C}_1$. Let $k \in \ZZ_n$ be such that $(j,k) \in \Delta_2$ (note that such a~$k$ exists since $G$ is transitive). Then for each $k' \in \ZZ_n$ such that $(j,k') \in \Delta_1$, the fact that $\Delta_1$ is self-paired implies that $(k',k)$ is an arc of~$\Lambda$, and thus $k' \in \mathcal{C}_1$. By an inductive argument and connectedness of $\Gamma$, we find that $s=2$. Moreover, $\Gamma$ is bipartite with bi-partition $(\mathcal{C}_1, \mathcal{C}_2)$. Since this contradicts the assumption that $\Gamma$ has an odd cycle, the underlying graph of $\Lambda$ is connected and thus $M_1M_2$ is irreducible.
\end{proof}

\begin{theorem} \label{theorem:K}
With reference to Agreement~\ref{notation}, let $\Gamma=\GG(\Delta_1, \Delta_2, \mathcal{I}, \Delta_4, \Delta_5)$ be a graph of odd order, and assume that $\mathcal{O}(\Delta_1)$ and $\mathcal{O}(\Delta_2)$ are non-singular. Then $\Gamma$ is a nut graph if and only if $\mathcal{O}(\Delta_4)$ and $\mathcal{O}(\Delta_5)$ are non-singular, $\mathcal{O}(\Delta_1 \cup \Delta_2)$ is connected and $\Gamma$ satisfies the valence condition.
\end{theorem}

\begin{proof}
Suppose first that $\Gamma$ is a nut graph. Since $\Gamma$ is connected, $\mathcal{O}(\Delta_1 \cup \Delta_2)$ is connected as well.
By Proposition~\ref{prop:o_vo_e}, $G \times H \times K$ is a subgroup of $\Aut(\Gamma)$ with two vertex and three edge orbits. Lemma~\ref{lemma:1} thus implies that $\Gamma$ satisfies the valence condition.
To complete the first part of the proof, we need to show that $\mathcal{O}(\Delta_i)$ is non-singular for each $i \in \{4,5\}$. Assume to the contrary, that  $\mathcal{O}(\Delta_4)$ is singular and let $\mathbf{x}$ be a kernel eigenvector of $\mathcal{O}(\Delta_4)$ (the consideration for the case when  $\mathcal{O}(\Delta_5)$ is singular is analogous). Define the labeling $\ell$ of $V(\Gamma)$ as follows. For each  $i \in \ZZ_n$, $j \in \ZZ_m$ and $j' \in \ZZ_t$, set $\ell(v_i^j)=\mathbf{x}(j)$ and $\ell(u_i^{j'})=0$. Observe that since $\mathcal{O}(\Delta_4)$ is regular, a double counting argument shows that $\sum_{j \in \ZZ_m}\mathbf{x}(j) = 0$. It is now easy to verify that the labeling~$\ell$ corresponds to a kernel eigenvector of~$\Gamma$. However, since this kernel eigenvector has some zero entries, $\Gamma$ is not a nut graph. Therefore, $\mathcal{O}(\Delta_4)$ is non-singular.

For the converse, recall that since $\Gamma$ satisfies the valence condition, it is singular by Lemma~\ref{lemma:singular} and Proposition~\ref{prop:o_vo_e}. Let $\mathbf x$ be a corresponding (kernel) eigenvector. For each $i \in \ZZ_n$, set 
\[ 
a_i = \sum_{j \in \ZZ_m}\mathbf{x}(v_i^j), \quad \alpha_i = \sum_{\oznaka{1}{i}{s}} a_s, \quad \textrm{and} \quad b_i = \sum_{j \in \ZZ_t}\mathbf{x}(u_i^j), \quad \beta_i = \sum_{\oznaka{2}{i}{q}}b_q.
\]

Consider the local conditions at vertices $v_i^j$, where $j \in \ZZ_m$, and ``sum" them for all~$j$. Since $\mathcal{O}(\Delta_4)$ has valence~$\kappa_4$, we have that
\[
      \sum_{j\in \ZZ_m}\left( \sum_{\oznaka{1}{i}{s}}\left( \sum_{\oznaka{4}{j}{r}} \mathbf{x}(v_s^r) \right) + b_i \right) = 
      \sum_{\oznaka{1}{i}{s}}\left(\sum_{j \in \ZZ_m}\left( \sum_{\oznaka{4}{j}{r}} \mathbf{x}(v_s^r) \right) \right) + \sum_{j \in \ZZ_m} b_i = 
\]

\[     =  \kappa_4 \sum_{\oznaka{1}{i}{s}}a_s+ mb_i = \kappa_4\alpha_i+mb_i= 0, 
\]

\noindent
and so 
\begin{equation}\label{K:cond1}
\alpha_i=-\frac{ mb_i}{\kappa_4}.
\end{equation}
Next, for $\oznaka{1}{i}{s}$, consider the local conditions at vertices $u_s^j$, where $j \in \ZZ_t$, and ``sum" them for all $j$ and~$s$. 
Since $\mathcal{O}(\Delta_5)$ has valence~$\kappa_5$, we have that

\[
     \sum_{\oznaka{1}{i}{s}}\left(\sum_{j \in \ZZ_t}\left(\sum_{\oznaka{2}{s}{q}}
    \left( \sum_{\oznaka{5}{j}{p}} \mathbf{x}(u_q^p)\right) + a_s \right)\right) =
\]

\[
    = \sum_{\oznaka{1}{i}{s}}\left(\sum_{\oznaka{2}{s}{q}} \kappa_5b_q + ta_s\right)
    = \sum_{\oznaka{1}{i}{s}}\left(\sum_{\oznaka{2}{s}{q}}\kappa_5 b_q\right) + t \alpha_i = 0.
\]

\noindent
Applying (\ref{K:cond1}), $k_1=\kappa_1\kappa_4$, $k_2=\kappa_2\kappa_5$ and then $k_1k_2=mt$, this is further equivalent to
\begin{equation} \label{K:cond3}
\begin{aligned}
\sum_{\oznaka{1}{i}{s}}\left(\sum_{\oznaka{2}{s}{q}} b_q\right) - \frac{mt\kappa_1\kappa_2}{k_1k_2} b_i =\sum_{\oznaka{1}{i}{s}}\left(\sum_{\oznaka{2}{s}{q}} b_q\right) - \kappa_1\kappa_2 b_i = 0.
\end{aligned}
\end{equation}
The system of these $n$ equations can be written in a matrix form as
\begin{equation}\label{K:E1}
\A(\mathcal{O}(\Delta_1))\A(\mathcal{O}(\Delta_2))\mathbf{b} = \kappa_1\kappa_2\mathbf{b}.
\end{equation}

Similarly, one can consider local conditions at vertices $u_i^j$, where $j \in \ZZ_t$, and ``sum" them for all~$j$, and the local conditions at vertices $v_q^j$, where $j \in \ZZ_m$ and $\oznaka{2}{i}{q}$, and ``sum" them for all $j$ and~$q$. This way we  obtain
\vspace{-10pt}
\[
\beta_i=-\frac{ta_i}{\kappa_5} \quad {\rm and} \quad 
 \sum_{\oznaka{2}{i}{q}}\left(\sum_{\oznaka{1}{q}{s}}\kappa_4 a_s\right) + m \beta_i = 0,
\]

\noindent
and hence 
\begin{equation} \label{K:E2}
\A(\mathcal{O}(\Delta_2))\A(\mathcal{O}(\Delta_1))\mathbf{a} = \kappa_1\kappa_2\mathbf{a}.
\end{equation}

\noindent
From (\ref{K:E1}) and (\ref{K:E2}) we conclude that $\mathbf{b}$ and $\mathbf{a}$, respectively, are eigenvectors of $\A(\mathcal{O}(\Delta_1))\A(\mathcal{O}(\Delta_2))$ 
and $\A(\mathcal{O}(\Delta_2))\A(\mathcal{O}(\Delta_1))$, respectively, corresponding to the eigenvalue $\kappa_1\kappa_2$. We now prove~that $\kappa_1\kappa_2$ has multiplicity~$1$ (in both cases), implying that $\mathbf{a}$ and $\mathbf{b}$ are unique up to scalar multiplication. To simplify the notation, set $M_1 = \A(\mathcal{O}(\Delta_1))$, $M_2=\A(\mathcal{O}(\Delta_2))$ and $M_1M_2=[m_{i,j}]_{0 \le i,j < n}$. Observe that $m_{i,j} \ge 0$ for all $i,j \in \ZZ_n$. Moreover, the sum of the entries in each row equals~$\kappa_1\kappa_2$ (note that the entry $m_{i,j}$ represents the number of elements $k \in \ZZ_n$ such that $(i,k) \in \Delta_1$ and $(k,j) \in \Delta_2$).
Thus, the eigenvalue $\kappa_1\kappa_2$ is the spectral radius of $M_1M_2$ (and of $M_2M_1$), and one of the corresponding eigenvectors is the all~$1^s$ eigenvector. 

For the uniqueness we apply the Perron-Frobenius theorem. Observe that although $M_1M_2$ is not necessarily a symmetric matrix, it is irreducible if and only if $M_2M_1=(M_1M_2)^T$ is irreducible. It thus suffices to consider only $M_1M_2$. 
Since $\mathcal{O}(\Delta_1 \cup \Delta_2)$ is connected by assumption and contains an odd cycle (recall that $n$ is odd and note that $\mathcal{O}(\Delta_1 \cup \Delta_2)$ is regular), $M_1M_2$ is irreducible by Lemma~\ref{lemma:matrix_product}. Therefore, the Perron-Frobenius theorem implies that~$\kappa_1\kappa_2$ has multiplicity~$1$, implying that $\mathbf{a}$ and $\mathbf{b}$ are uniquely determined up to scalar multiplication. We can thus assume that $\mathbf{b}=(1,1,\ldots,1)$ and $\mathbf{a}=(a,a,\ldots,a)$ for some~$a$. By (\ref{K:cond1}), we have that $\alpha_i=-\frac{m}{\kappa_4}$, and consequently

\[
-\frac{m}{\kappa_4}=\alpha_i=\sum_{\oznaka{1}{i}{s}}a_s=\kappa_1a,
\]
which implies that $-m=\kappa_1\kappa_4a=k_1 a$. Hence, $a=-\frac{m}{k_1} \ne 0$.

Finally, considering the local conditions at all $v \in V$ and at all $u \in U$, we find that
\begin{equation} \label{K:cond20}
\begin{split}
\A(\Gamma_1) \mathbf{x_v} = -\mathbf{1}, \quad &\textrm{where} \quad  \mathbf{x_v}=\left(x(v_0^0), x(v_1^0), \ldots, x(v_{n-1}^{m-1})\right), \\
\A(\Gamma_2) \mathbf{x_u} = -a\mathbf{1} =\frac{m}{k_1}\mathbf{1}, \quad &\textrm{where} \quad  \mathbf{x_u}=\left(x(u_0^0), x(u_1^0), \ldots, x(u_{n-1}^{t-1})\right).
\end{split}
\end{equation}
We thus obtain two non-homogeneous systems of linear equations. Since $\mathcal{O}(\Delta_i)$ is non-singular for each $i \in \{1,2,4,5\}$, Proposition~\ref{prop:direct_product} implies that $\Gamma_1$ and $\Gamma_2$ are also non-singular. Each of the systems in (\ref{K:cond20}) thus has a unique solution. Therefore, the null space of~$\Gamma$ is one-dimensional, and so the proof of Lemma~\ref{lemma:singular} implies that the corresponding eigenvectors have no zero entries. Consequently, $\Gamma$ is a nut graph of order $n(m+t)$.
\end{proof}

\section{Infinite families}\label{sec:4} 

Graphs arising from Construction~\ref{cons:merge} that satisfy the assumptions of Theorem~\ref{theorem:K} to be nut graphs, provide strong evidence in support of Conjecture~\ref{conjecture}. To illustrate the power of this construction, we present a number of infinite classes of examples, each of which yields infinitely many orders for which a nut graph with $(o_v,o_e)=(2,3)$ exists (see Corollary~\ref{cor:3n}, Corollary~\ref{cor:tetra} and Corollary~\ref{cor:novi}). 

To derive these results we consider some special types of examples of graphs arising from Theorem~\ref{theorem:K}. In particular, in Proposition~\ref{prop:general} we focus on examples for which either $\Delta_2=\Delta_1$ or $\Delta_2=\mathcal{I}$. Moreover, we only consider very specific types of (multi)graphs resulting from~$\Delta_4$ and~$\Delta_5$.
Before describing these types, we observe the following. For any transitive permutation group $G \le S_n$ and a corresponding self-paired orbital $\Delta$, the (multi)graph $\Lambda= \mathcal{O}(\Delta)$ is arc-transitive (of course, $G$ might be a proper subgroup of $\Aut(\Lambda)$). Conversely, if $\Lambda$ is an arc-transitive (multi)graph of order~$n$, then $\Aut(\Lambda) \le S_n$ is transitive and there exists a corresponding self-paired orbital~$\Delta$ such that $\Lambda= \mathcal{O}(\Delta)$. 
It thus suffices to choose a non-singular arc-transitive (multi)graph~$\Lambda_1$ (yielding~$G=\Aut(\Lambda_1)$ and~$\Delta_1$), non-singular arc-transitive (multi)graphs~$\Lambda_4$ and~$\Lambda_5$ (yielding~$H$, $\Delta_4$, and $K$, $\Delta_5$, respectively), and to decide whether $\Delta_2=\Delta_1$ or $\Delta_2=\mathcal{I}$. 
Note that if $\Delta_2=\mathcal{I}$, then $\mathcal{O}(\Delta_2)$ consists of $n$ vertices with loops, that is, $\mathcal{O}(\Delta_2) = n\zanka$, where $\zanka$ denotes a one vertex multigraph with a loop. Of course, in the end we still need to make sure that all the assumptions and conditions of Theorem~\ref{theorem:K} hold.

We now describe the above mentioned types of (multi)graphs $\Lambda_4$ and $\Lambda_5$, that we will consider in Proposition~\ref{prop:general}. Each of them will be composed as a direct product of the following arc-transitive (multi)graphs: a multigraph consisting of a number of loops, non-singular arc-transitive circulants of prime orders, and complete graphs.
Before proceeding, we state a result which assures that for each odd prime~$p$ and each even divisor~$d$ of~$p-1$, a non-singular arc-transitive (connected) circulant of order~$p$ and valence~$d$ exists.

\begin{proposition}\label{prop:asdfghjkl}
Let $p$ be an odd prime and $d$ be an even divisor of~$p-1$. Then there exists a non-singular arc-transitive connected circulant of order~$p$ and valence~$d$.
\end{proposition}

\begin{proof}
It is well known that the group $\ZZ_p^*$ of units of~$\ZZ_p$ is a cyclic group of order~$p-1$. For each divisor~$d$ of~$p-1$ there thus exists a unique (cyclic) subgroup~$S \le \ZZ_p^*$ of order~$d$. Since in our case~$d$ is even, $-1 \in S$ and so $S=-S$. It is well known that in this case the circulant $\Cay(\ZZ_p; S)$ is arc-transitive. 
Using the Leibniz formula it is not difficult to see that for any prime~$p$, a $p \times p$ circulant $\{0,1\}$-matrix is singular if and only if all its entries are equal, i.e., it is the zero or all ones matrix (but see for instance \cite{KS12}). Consequently, $\Cay(\ZZ_p; S)$ is non-singular, arc-transitive, connected and of valence~$d$.
\end{proof}

{We make the convention to denote the circulant $\Cay(\ZZ_p; S)$ of order~$p$ and valence~$d=|S|$ from the proof of Proposition~\ref{prop:asdfghjkl} by $\mathcal{C}_p^d$. Using this convention, each of $\Lambda_4$ and $\Lambda_5$ will be a (multi)graph of the form }
\begin{equation}\label{eq:graf}
    z_0 \zanka \times \mathcal{C}_{z_1}^{d_1} \times \cdots \times \mathcal{C}_{z_r}^{d_r} \times K_{z'_1}\times \cdots \times K_{z'_q},
\end{equation}
for some integers $r,q \ge 0$ and $z_0 \ge 1$, odd primes $z_1, \ldots,z_r$, together with even integers $d_1,\ldots,d_r$ such that $d_i | (z_i-1)$ for all~$i \in \{1,\ldots, r\}$, and integers $z_1', \ldots, z_q'\ge 2$. 

The multigraph $\zanka$ is clearly non-singular, and so Proposition~\ref{prop:direct_product} and Proposition~\ref{prop:asdfghjkl} imply that each of $\Lambda_4$ and $\Lambda_5$ is non-singular and arc-transitive. Note that if $\Lambda_4$ (or $\Lambda_5$) is as in (\ref{eq:graf}), then its order and valence are  

\begin{equation}\label{eq:fact}
z=z_0 \left(\prod_{i=1}^rz_i\right) \left(\prod_{i=1}^qz'_i \right) \quad {\rm and} \quad
\left(\prod_{i=1}^{r}d_i\right) \left(\prod_{i=1}^q(z'_i-1)\right),
\end{equation}
respectively. We make the usual agreement that $\prod_{i=1}^0x_i=1$.

\begin{proposition}\label{prop:general}
Let $m \ge 1$ be an odd integer and $t \ge 2$ be an even integer such that
\begin{equation}\label{eq:P1def}
m= m_0 \left(\prod_{i=1}^{r_1} m_i\right) \left(\prod_{i=1}^{q_1}m_i'\right) \quad {\rm and} \quad t= t_0 \left(\prod_{i=1}^{r_2} t_i\right) \left(\prod_{i=1}^{q_2}t_i'\right),
\end{equation}
where $r_1,q_1,r_2, q_2 \ge 0$, $m_0, t_0 \ge 1$, and where
$m_1, \ldots, m_{r_1}, t_1, \ldots, t_{r_2}$ are odd primes and $m_1', \ldots, m_{q_1}'$, $t_1', \ldots, t_{q_2}' \ge 2$. Suppose that for each~$i$, $1 \le i \le r_1$, and each $j$, $1 \le j \le r_2$, there exists an even divisor $d_i$ of $m_i-1$ and an even divisor $d_j'$ of $t_j-1$, together with an even integer~$a \ge 2$ and a~$\kappa \in \{a,a^2\}$,  such that
\begin{equation}\label{eq:P1val}
    \kappa\left(\prod_{i=1}^{r_1}d_i\right)\left(\prod_{j=1}^{r_2}d_j'\right) \left(\prod_{i=1}^{q_1}(m_i'-1)\right) \left(\prod_{i=1}^{q_2}(t_i'-1)\right) = mt 
\end{equation}
and
\begin{equation}\label{eq:P1(2,3)}
    t+a\cdot \left(\prod_{i=1}^{r_1}d_i\right) \left(\prod_{i=1}^{q_1}(m_i'-1)\right) \ne m+\frac{\kappa}{a}\cdot \left(\prod_{j=1}^{r_2}d_j'\right)\left(\prod_{i=1}^{q_2}(t_i'-1)\right).
\end{equation}
If there exists a non-singular arc-transitive connected graph of odd order~$n \ge 3$ and valence~$a$, then there exists a nut graph with $(o_v,o_e)=(2,3)$ of order $n(m+t)$.
\end{proposition}

\begin{proof}
Set \[\Lambda_4 = m_0 \zanka \times \mathcal{C}_{m_1}^{d_1} \times \cdots \times \mathcal{C}_{m_{r_1}}^{d_{r_1}} \times K_{m'_1}\times \cdots \times K_{m'_{q_1}} \quad {\rm and} \] 

\[\Lambda_5 = t_0 \zanka \times \mathcal{C}_{t_1}^{d_1'} \times \cdots \times \mathcal{C}_{t_{r_2}}^{d_{r_2}'} \times K_{t'_1}\times \cdots \times K_{t'_{q_2}}.\]

Suppose that $\Lambda_1$ is a non-singular arc-transitive connected graph of odd order~$n \ge 3$ and valence~$a$. If $\kappa = a$, set $\Lambda_2=n\zanka$, while if $\kappa = a^2$, set $\Lambda_2 = \Lambda_1$. 
We claim that the groups $G=\Aut(\Lambda_1)$, $H=\Aut(\Lambda_4)$, $K=\Aut(\Lambda_5)$, and the corresponding orbitals $\Delta_i$, $i \in \{1,2,4,5\}$, satisfy all the assumptions and conditions of Construction~\ref{cons:merge} and Theorem~\ref{theorem:K}. This will imply that the resulting graph $\Gamma= \GG(\Delta_1, \Delta_2, \mathcal{I}, \Delta_4, \Delta_5)$ is a nut graph.

By the remarks preceding this proposition, each of  $\Lambda_i$, $i \in \{4,5\}$, is non-singular, while each of $\Lambda_i$, $i \in \{1,2\}$, is non-singular by assumption. Since $\Lambda_1$ is connected, $\mathcal{O}(\Delta_1 \cup \Delta_2)$ is also connected. Suppose that $\Delta_2=\mathcal{I}$ (which occurs if and only if $\kappa = a$). If in addition $\Delta_5=\mathcal{I}$ (which occurs if and only if $t=t_0$ and $r_2=q_2=0$), then we can change the factorization of~$t$ from~(\ref{eq:P1def}) by substituting $t_0$ by ${t_0}/{2}$, and setting $q_2=1$ and $t_1'=2$. Note that doing this we do not change the conditions (\ref{eq:P1val}) and (\ref{eq:P1(2,3)}). We can thus assume that at least one of $\Delta_2$ and $\Delta_5$ is non-diagonal.

To complete the proof of our claim, we verify that the valence condition from Theorem~\ref{theorem:K} holds. Observe that $\Lambda_1 \times \Lambda_4$ and $\Lambda_2 \times \Lambda_5$ are of valence
\vspace{-5pt}
\begin{equation}\label{eq:P1proof}
k_1 = a\cdot \left(\prod_{i=1}^{r_1}d_i\right) \left(\prod_{i=1}^{q_1}(m_i'-1)\right) \quad {\rm and}  \quad k_2 =\frac{\kappa}{a}\cdot \left(\prod_{j=1}^{r_2}d_j'\right)\left(\prod_{i=1}^{q_2}(t_i'-1)\right),
\end{equation}

respectively. By (\ref{eq:P1val}), our claim is proved, and so $\Gamma$ is indeed a nut graph. 

We finally show that $\Gamma$ is a graph with $(o_v,o_e)=(2,3)$. Note that by Proposition~\ref{prop:o_vo_e}, it suffices to show that $\Gamma$ is not regular. Indeed, $\Gamma$ has $nm$ vertices of valence $t+k_1$ and $nt$ vertices of valence $m+k_2$, where $k_1$ and $k_2$ are as in (\ref{eq:P1proof}). By the assumption (\ref{eq:P1(2,3)}), $t+k_1 \ne m+k_2$, which completes the proof.
\end{proof}

We now present three corollaries of Proposition~\ref{prop:general}, in which we consider different possibilities for the parameter~$a$. These correspond to the situations where the graph~$\Lambda_1$ from the proof of Proposition~\ref{prop:general} is a cycle (Corollary~\ref{cor:3n}), an arc-transitive tetravalent graph (Corollary~\ref{cor:tetra}), and an arc-transitive circulant of prime order (Corollary~\ref{cor:novi}), respectively.
In the first and the second corollary, we also provide several small odd integers~$s$ that can be written in the form $s = m + t$, such that $m$ and $t$ admit factorizations of the form~(\ref{eq:P1def}), satisfying the conditions (\ref{eq:P1val}) and~(\ref{eq:P1(2,3)}) of Proposition~\ref{prop:general} {for appropriate even divisors $d_i \mid (m_i-1)$, $1 \le i \le r_1$, and $d_j'\mid (t_j-1)$, $1 \le j \le r_2$, and an appropriate~$\kappa \in \{a,a^2\}$.} For each such $s$, a corresponding~$m$ and~$t$, appropriate factorizations of~$m$ and~$t$ from~(\ref{eq:P1def}), and choices for $d_i$ and $d_j'$, are given in Table~\ref{table:3n} and Table~\ref{table:tetra}, respectively. {We denote the factorization of an integer~$z$ from~(\ref{eq:fact}), together with the choice for the even divisors $d_i \mid (z_i-1)$, by~$[z]$. Moreover, we represent it as $[z]=z_0 | z_1^{(d_1)}, \ldots, z_r^{(d_r)} | z'_1, \ldots, z'_q$}. The parameters $k_1$ and $k_2$ from the tables are as in the proof of Proposition~\ref{prop:general}. As we will see, the set of the small~$s$ from Corollary~\ref{cor:3n} and Corollary~\ref{cor:tetra} suffices to confirm Conjecture~\ref{conjecture} for all but~$17$ orders up to~$2\,500$. The remaining~$17$ orders are covered by Corollary~\ref{cor:novi}.

\begin{corollary}\label{cor:3n}
Let $n \ge 3$ be an odd integer. Let $m$ and~$t$ be integers admitting factorizations of the form~(\ref{eq:P1def}) such that the conditions (\ref{eq:P1val}) and (\ref{eq:P1(2,3)}) from Proposition~\ref{prop:general} hold {for some $d_i$, $1 \le i \le r_1$, and $d_j'$, $1 \le j \le r_2$, and} for $a=2$ and some $\kappa \in \{2,4\}$. Then there exists a nut graph with $(o_v,o_e)=(2,3)$ of order $n(m+t)$. 
In particular, there exists a nut graph with $(o_v,o_e)=(2,3)$ of order $ns$ for each odd $n \ge 3$ and \[s \in \{3, 5, 7, 11, 13, 29, 31, 37, 41, 47, 53, 59, 83, 101, 103, 109, 127, 131, 137, 139\}.\]

\end{corollary}

\begin{proof}
Observe that the cycle~$C_n$ is connected, arc-transitive and non-singular (recall that $n$ is odd). The result now follows from Proposition~\ref{prop:general} and Table~\ref{table:3n}. For instance, if $s=109=45+64$, we can choose {$[m]=1|5^{(2)}|9$} and $[t]=1||4,16$. In other words, we can take $m_0=1$, $r_1=1$, $q_1=1$, $m_1=5$, {$d_1=2$,} $m_1'=9$, and $t_0=1$, $r_2=0$, $q_2 = 2$, $t_1'=4$, $t_2'=16$. Choosing $\kappa=4$, it is straightforward to check that the conditions (\ref{eq:P1val}) and (\ref{eq:P1(2,3)}) are satisfied. Proposition~\ref{prop:general} therefore implies that there exists a nut graph with $(o_v,o_e)=(2,3)$ of order~$109n$ for each odd~$n \ge 3$. The other possibilities for~$s$ can be checked analogously.
\end{proof}

\begin{table}[h!]
\centering
\renewcommand{\arraystretch}{1.0}
\begin{adjustbox}{width=0.75\textwidth}
\begin{tabular}{C{1.5cm}|C{1cm}ccC{1cm}|C{2cm}|C{2cm}|cc}
\toprule
order & $s$ & $m$ & $t$ & $\kappa$ & $[m]$ & $[t]$ & $k_1$ & $k_2$ \\
\midrule
$3n$ & $3$ & $1$ & $2$ & $2$ & $1||$ & $1||2$ & $2$ & $1$ \\
$5n$ & $5$ & $1$ & $4$ & $4$ & $1||$ & $4||$ & $2$ & $2$ \\
$7n$ & $7$ & $3$ & $4$ & $2$ & $1||3$ & $1||4$ & $4$ & $3$ \\
$11n$ & $11$ & $3$ & $8$ & $4$ & $1||3$ & $2||4$ & $4$ & $6$ \\
$13n$ & $13$ & $1$ & $12$ & $4$ & $1||$ & $3||4$ & $2$ & $6$ \\
$29n$ & $29$ & $21$ & $8$ & $4$ & $3||7$ & $1||8$ & $12$ & $14$ \\
$31n$ & $31$ & $15$ & $16$ & $4$ & $3||5$ & $1||16$ & $8$ & $30$ \\
$37n$ & $37$ & $21$ & $16$ & $4$ & $1||3,7$ & $2||8$ & $24$ & $14$ \\
$41n$ & $41$ & $9$ & $32$ & $4$ & $1||9$ & $2||4,4$ & $16$ & $18$  \\
$47n$ & $47$ & $15$ & $32$ & $4$ & $1||3,5$ & $2||16$ & $16$ & $30$  \\
$53n$ & $53$ & $5$ & $48$ & $4$ & $1||5$ & $3||16$ & $8$ & $30$ \\
$59n$ & $59$ & $3$ & $56$ & $4$ & $3||$ & $1||7,8$ & $2$ & $84$ \\
$83n$ & $83$ & $3$ & $80$ & $2$ & $1||3$ & $1||5,16$ & $4$ & $60$ \\
$101n$ & $101$ & $5$ & $96$ & $4$ & $1||5$ & $2||3,16$ & $8$ & $60$ \\
$103n$ & $103$ & $55$ & $48$ & $2$ & $1||5,11$ & $1||4,12$ & $80$ & $33$ \\
$109n$ & $109$ & $45$ & $64$ & $4$ & {$1|5^{(2)}|9$} & $1||4,16$ & $32$ & $90$ \\
$127n$ & $127$ & $63$ & $64$ & $4$ & {$1|7^{(2)}|9$} & $1||64$ & $32$ & $126$ \\
$131n$ & $131$ & $35$ & $96$ & $4$ & $1||5,7$ & $2||6,8$ & $48$ & $70$ \\
$137n$ & $137$ & $105$ & $32$ & $2$ & $1||5,21$ & $1||4,8$ & $160$ & $21$ \\
{$139n$} & $139$ & $19$ & $120$ & $4$ & $1|19^{(6)}|$ & $1||6,20$ & $12$ & $190$ \\

\bottomrule
\end{tabular}
\end{adjustbox}
\caption{Appropriate $m$ and $t$ with $[m]$, $[t]$, for all the $s$ from Corollary~\ref{cor:3n}.}
\label{table:3n}
\end{table}

\begin{figure}[h!] 
\centering
\begin{tikzpicture} [scale=0.3]
\def\rzeroa{2cm}  
\def\ronea{4.5cm}   
\def\roneb{5.5cm}
\def\ronec{6.5cm}
\def\roned{7.5cm}

\def\n{11}

\def\colorA{black}
\def\colorB{black}

\foreach \i in {0,...,\numexpr\n-1} {
    \def\angle{360/\n * \i}

    \node[draw, circle, fill=\colorA, minimum size=4pt, inner sep=1.5pt] (zero \i a) at ({\angle}:\rzeroa) {};  

    \node[draw, circle, fill=\colorB, minimum size=4pt, inner sep=1.5pt] (one \i a) at ({\angle}:\ronea) {};
    \node[draw, circle, fill=\colorB, minimum size=4pt, inner sep=1.5pt] (one \i b) at ({\angle}:\roneb) {};
    \node[draw, circle, fill=\colorB, minimum size=4pt, inner sep=1.5pt] (one \i c) at ({\angle}:\ronec) {};
    \node[draw, circle, fill=\colorB, minimum size=4pt, inner sep=1.5pt] (one \i d) at ({\angle}:\roned) {};
}

\foreach \i in {0,...,\numexpr\n-1} {
    \pgfmathsetmacro\ipone{int(mod(\i+1, \n))}   

        \draw[black, thin] (zero \i a) to (zero \ipone a);

    \foreach \k in {a,b,c,d} {
        \draw[black, thin] (one \i \k) to (one \ipone \k);
    }
}

\foreach \i in {0,...,\numexpr\n-1} {
    \def\angle{360/\n * \i}
    \node[draw=blue!60, semithick, rounded corners=8pt, dash pattern=on 3pt off 2pt, circle, minimum size=9pt] 
        (blueRectangle\i) at ({\angle}: \rzeroa) {}; 
    \node[draw=magenta!60, semithick, rounded corners=8pt, dash pattern=on 3pt off 2pt, rectangle, minimum width=4.5cm, minimum height=2cm, rotate=\angle, transform shape] 
        (magentaRectangle\i) at ({\angle}: 6cm) {}; 
    \draw[line width=4pt, gray] (blueRectangle\i) -- (magentaRectangle\i); 
}

\end{tikzpicture}
\begin{tikzpicture} [scale=0.3] 
\def\rzeroa{2cm}  
\def\rzerob{3cm}
\def\rzeroc{4cm}
\def\ronea{6.1cm}   
\def\roneb{7.5cm}

\def\n{11}

\def\colorA{black}
\def\colorB{black}

\foreach \i in {0,...,\numexpr\n-1} {
    \def\angle{360/\n * \i}

    \node[draw, circle, fill=\colorA, minimum size=4pt, inner sep=1.5pt] (zero \i a) at ({\angle}:\rzeroa) {};
    \node[draw, circle, fill=\colorA, minimum size=4pt, inner sep=1.5pt] (zero \i b) at ({\angle}:\rzerob) {};
    \node[draw, circle, fill=\colorA, minimum size=4pt, inner sep=1.5pt] (zero \i c) at ({\angle}:\rzeroc) {};

    \node[draw, circle, fill=\colorB, minimum size=4pt, inner sep=1.5pt] (one \i a) at ({\angle - 8}:\ronea) {};
    \node[draw, circle, fill=\colorB, minimum size=4pt, inner sep=1.5pt] (two \i a) at ({\angle + 8}:\ronea) {};
    \node[draw, circle, fill=\colorB, minimum size=4pt, inner sep=1.5pt] (one \i b) at ({\angle - 7}:\roneb) {};
    \node[draw, circle, fill=\colorB, minimum size=4pt, inner sep=1.5pt] (two \i b) at ({\angle + 7}:\roneb) {};
}

\foreach \i in {0,...,\numexpr\n-1} {
    \pgfmathsetmacro\ipone{int(mod(\i+1, \n))}   

    \foreach \k in {a,b,c} {
        \foreach \l in {a,b,c} {
        \ifx\k\l\else 
        \draw[gray, thin] (zero \i \k) to (zero \ipone \l);
    \fi    
    }
    }

    \draw[] (one \i a) -- (two \i a);
    \draw[] (one \i a) -- (one \i b);
    \draw[] (one \i a) -- (two \i b);
    \draw[] (two \i a) -- (one \i b);
    \draw[] (two \i a) -- (two \i b);
    \draw[] (one \i b) -- (two \i b);

\foreach \i in {0,...,\numexpr\n-1} {
    \def\angle{360/\n * \i}
    \node[draw=blue!60, thin, rounded corners=8pt, dash pattern=on 3pt off 2pt, rectangle, minimum width=3cm, minimum height=1.6cm, rotate=\angle, transform shape] 
        (blueRectangle\i) at ({\angle}: \rzerob) {}; 
    \node[draw=magenta!60, thin, rounded corners=8pt, dash pattern=on 3pt off 2pt, rectangle, minimum width=3cm, minimum height=3cm, rotate=\angle, transform shape] 
        (magentaRectangle\i) at ({\angle}: 6.8cm) {}; 
    \draw[line width=4pt, gray] (blueRectangle\i) -- (magentaRectangle\i); 
}

}

\end{tikzpicture}
\begin{tikzpicture} [scale=0.3] 
\def\rzeroa{3cm}  
\def\ronea{5.1cm}   
\def\roneb{5.7cm}
\def\ronec{6.3cm}
\def\roned{6.9cm}
\def\ronee{7.5cm}

\def\n{11}

\def\colorA{black}
\def\colorB{black}

\foreach \i in {0,...,\numexpr\n-1} {
    \def\angle{360/\n * \i}

    \node[draw, circle, fill=\colorA, minimum size=4pt, inner sep=1.5pt] (zero \i a) at ({\angle}:\rzeroa) {}; 

    \node[draw, circle, fill=\colorB, minimum size=3pt, inner sep=1pt] (one \i a) at ({\angle - 8}:\ronea) {};
    \node[draw, circle, fill=\colorB, minimum size=3pt, inner sep=1pt] (two \i a) at ({\angle + 8}:\ronea) {};
    \node[draw, circle, fill=\colorB, minimum size=3pt, inner sep=1pt] (one \i b) at ({\angle - 7}:\roneb) {};
    \node[draw, circle, fill=\colorB, minimum size=3pt, inner sep=1pt] (two \i b) at ({\angle + 7}:\roneb) {};
    \node[draw, circle, fill=\colorB, minimum size=3pt, inner sep=1pt] (one \i c) at ({\angle - 6.2}:\ronec) {};
    \node[draw, circle, fill=\colorB, minimum size=3pt, inner sep=1pt] (two \i c) at ({\angle + 6.2}:\ronec) {};  
    \node[draw, circle, fill=\colorB, minimum size=3pt, inner sep=1pt] (one \i d) at ({\angle - 5.6}:\roned) {};
    \node[draw, circle, fill=\colorB, minimum size=3pt, inner sep=1pt] (two \i d) at ({\angle + 5.6}:\roned) {};    
    \node[draw, circle, fill=\colorB, minimum size=3pt, inner sep=1pt] (one \i e) at ({\angle - 5}:\ronee) {};
    \node[draw, circle, fill=\colorB, minimum size=3pt, inner sep=1pt] (two \i e) at ({\angle + 5}:\ronee) {};
}

\foreach \i in {0,...,\numexpr\n-1} { 
    \foreach \j in {0,...,\numexpr\n-1} {
        \ifx\i\j\else 
        \draw[gray, thin] (zero \i a) to (zero \j a);
    \fi    
    }
    \foreach \k in {a,b,c,d,e}{
        \draw[] (one \i \k) -- (two \i \k);
    }
}

\foreach \i in {0,...,\numexpr\n-1} {
    \def\angle{360/\n * \i}
    \node[draw=blue!60, semithick, rounded corners=8pt, dash pattern=on 3pt off 2pt, circle, minimum size=10pt] 
        (blueRectangle\i) at ({\angle}: \rzeroa) {}; 
    \node[draw=magenta!60, semithick, rounded corners=8pt, dash pattern=on 3pt off 2pt, rectangle, minimum width=3.5cm, minimum height=2.7cm, rotate=\angle, transform shape] 
        (magentaRectangle\i) at ({\angle}: \ronec) {}; 
    \draw[line width=4pt, gray] (blueRectangle\i) -- (magentaRectangle\i); 
}

\end{tikzpicture}
\caption{Examples of nut graphs with $(o_v,o_e)=(2,3)$ from Construction~\ref{cons:merge} of orders $5n$, $7n$ and $n^2$ for $n=11$. {The first and the second graph arise from Corollary~\ref{cor:3n} (for $s=5$ and $s=7$, respectively).}}
\label{pic:3n} 
\end{figure}
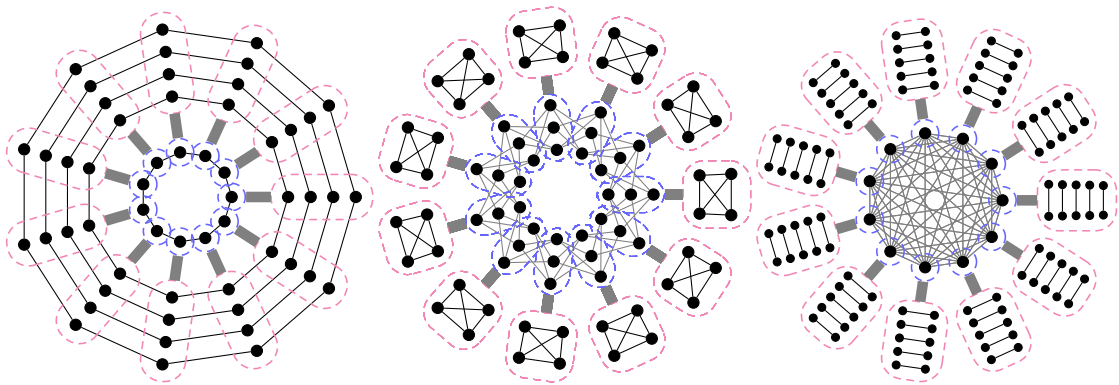

\begin{corollary}\label{cor:tetra}
Let $n \ge 5$ be an odd integer such that either $\modulo{n}{1}{4}$ or $n$ is not a prime. Let $m$ and~$t$ be integers admitting factorizations of the form~(\ref{eq:P1def}) such that the conditions (\ref{eq:P1val}) and (\ref{eq:P1(2,3)}) from Proposition~\ref{prop:general} hold {for some $d_i$, $1 \le i \le r_1$, and $d_j'$, $1 \le j \le r_2$, and} for $a=4$ and some $\kappa \in \{4,16\}$. Then there exists a nut graph with $(o_v,o_e)=(2,3)$ of order $n(m+t)$. 
In particular, there exists a nut graph with $(o_v,o_e)=(2,3)$ of order $ns$ for each $s \in \{17, 19, 23, 73, 79, 97, 107, 113\}$ and $n$ as above.
\end{corollary}

\begin{proof}
Let $n$ be as in the corollary. We claim that there exists a non-singular arc-transitive connected tetravalent graph~$\Lambda_1$ of order~$n$.
If $n$ is prime, then $\modulo{n}{1}{4}$, and so Proposition~\ref{prop:asdfghjkl} applies.
Otherwise, $n = n_1n_2$ for some $n_1, n_2 \ge 3$. By the remarks preceding Proposition~\ref{prop:general}, we can then simply take $\Lambda_1=C_{n_1} \times C_{n_2}$ (recall that since $n_1$ and $n_2$ are odd, $\Lambda_1$ is connected). Consequently, our claim holds, and so we can apply Proposition~\ref{prop:general}. For the second part we can proceed as in the proof of Corollary~\ref{cor:3n} using Table~\ref{table:tetra}. 
\end{proof}

\begin{table}[h!]
\centering
\renewcommand{\arraystretch}{1.0}
\begin{adjustbox}{width=0.7\textwidth}
\begin{tabular}{C{1.5cm}|C{1cm}ccC{1cm}|C{2cm}|C{2cm}|cc}
\toprule
order & $s$ & $m$ & $t$ & $\kappa$ & $[m]$ & $[t]$ & $k_1$ & $k_2$ \\
\midrule
$17n$ & $17$ & $1$ & $16$ & $16$ & $1||$ & $16||$ & $4$ & $4$ \\
$19n$ & $19$ & $3$ & $16$ & $16$ & $3||$ & $4||4$ & $4$ & $12$ \\
$23n$ & $23$ & $7$ & $16$ & $16$ & $7||$ & $2||8$ & $4$ & $28$ \\
$73n$ & $73$ & $9$ & $64$ & $16$ & $1||3,3$ & $4||4,4$ & $16$ & $36$ \\
$79n$ & $79$ & $15$ & $64$ & $16$ & $3||5$ & $4||16$ & $16$ & $60$ \\
$97n$ & $97$ & $1$ & $96$ & $16$ & $1||$ & $8||3,4$ & $4$ & $24$ \\
$107n$ & $107$ & $11$ & $96$ & $16$ & $1|11^{(2)}|$ & $2||4,12$ & $8$ & $132$ \\
$113n$ & $113$ & $49$ & $64$ & $16$ & {$1|7^{(2)},7^{(2)}|$} & $1||8,8$ & $16$ & $196$ \\
\bottomrule
\end{tabular}
\end{adjustbox}
\caption{Appropriate $m$ and $t$ with $[m]$, $[t]$, for all the $s$ from Corollary~\ref{cor:tetra}. }
\label{table:tetra}
\end{table}

Up to~$2\,500$, there are~$883$ odd non-prime integers. As already mentioned, Corollary~\ref{cor:3n} and Corollary~\ref{cor:tetra} provide examples of nut graphs with $(o_v,o_e)=(2,3)$ for~$866$ of these orders. More specifically, Corollary~\ref{cor:3n} provides examples for~$847$ of them, while Corollary~\ref{cor:tetra} provides examples for~$19$ additional ones. We take care of the remaining~$17$ orders using the following corollary of Proposition~\ref{prop:asdfghjkl} and Proposition~\ref{prop:general}.

\begin{corollary}\label{cor:novi}
Let $n \ge 3$ be a prime. Let $m$ and~$t$ be integers as in Proposition~\ref{prop:general} such that (\ref{eq:P1val}) and (\ref{eq:P1(2,3)}) hold {for some $d_i$, $1 \le i \le r_1$, and $d_j'$, $1 \le j \le r_2$, and} for some even divisor $a$ of $n-1$ and some $\kappa \in \{a,a^2\}$. Then there exists a nut graph with $(o_v,o_e)=(2,3)$ of order $n(m+t)$. 
\end{corollary}

\begin{table}[h!]
\centering
\renewcommand{\arraystretch}{1.0}
\begin{adjustbox}{width=0.80\textwidth}
\begin{tabular}{C{1.5cm}|C{1cm}cC{1cm}|C{1cm}C{1cm}|C{2cm}|C{2cm}|cc}
\toprule
order & $n$ & $m$ & $t$ & $a$ & $\kappa$ & $[m]$ & $[t]$ & $k_1$ & $k_2$ \\
\midrule
$361$ & $19$ & $1$ & $18$ & $18$ & $18$ & $1||$ & $9||2$ & $18$ & $1$ \\
$437$ & $19$ & $5$ & $18$ & $18$ & $18$ & $5||$ & $3||6$ & $18$ & $5$ \\
$529$ & $23$ & $1$ & $22$ & $22$ & $22$ & $1||$ & $11||2$ & $22$ & $1$ \\
$731$ & $43$ & $3$ & $14$ & $42$ & $42$ & $3||$ & $7||2$ & $42$ & $1$ \\
$817$ & $43$ & $7$ & $12$ & $42$ & $42$ & {$1|7^{(2)}|$} & $6||2$ & $84$ & $1$ \\
$989$ & $43$ & $21$ & $2$ & $42$ & $42$ & $21||$ & $1||2$ & $42$ & $1$ \\
$1\,139$ & $67$ & $11$ & $6$ & $66$ & $66$ & $11||$ & $3||2$ & $66$ & $1$ \\
$1\,207$ & $71$ & $7$ & $10$ & $70$ & $70$ & $7||$ & $5||2$ & $70$ & $1$ \\
{$1273$} & $19$ & $55$ & $12$ & $6$ & $6$ & $5||11$ & $1||12$ & $60$ & $11$ \\
$1\,349$ & $71$ & $5$ & $14$ & $70$ & $70$ & $5||$ & $7||2$ & $70$ & $1$ \\
$1\,501$ & $79$ & $13$ & $6$ & $78$ & $78$ & $13||$ & $3||2$ & $78$ & $1$ \\
$1\,541$ & $23$ & $55$ & $12$ & $22$ & $22$ & $5||11$ & $3||4$ & $220$ & $3$ \\
$1\,633$ & $23$ & $55$ & $16$ & $22$ & $22$ & $1||5,11$ & $8||2$ & $880$ & $1$ \\
$1\,817$ & $23$ & $55$ & $24$ & $22$ & $22$ & $11||5$ & $1||4,6$ & $88$ & $15$ \\
$1\,849$ & $43$ & $1$ & $42$ & $42$ & $42$ & $1||$ & $21||2$ & $42$ & $1$ \\
$2\,033$ & $19$ & $27$ & $80$ & $18$ & $18$ & $3||3,3$ & {$1|5^{(2)}|16$} & $72$ & $30$ \\
$2\,461$ & $23$ & $63$ & $44$ & $22$ & $22$ & $9||7$ & $2||22$ & $132$ & $21$ \\
\bottomrule
\end{tabular}
\end{adjustbox}
\caption{Appropriate $n$, {$a$ and $\kappa$, }and $m$, $t$ with $[m]$, $[t]$, for all the orders up to~$2\,500$ not covered by Corollaries~\ref{cor:3n} and~\ref{cor:tetra}.}
\label{table:others}
\end{table}

The $17$ odd non-prime orders up to~$2\,500$ which are not covered by Corollary~\ref{cor:3n} and Corollary~\ref{cor:tetra}, are the~$17$ integers from the first column of Table~\ref{table:others}. In this table, we provide appropriate integers $n$, $m$, $t$, $a$ and $\kappa$, together with suitable~$[m]$ and~$[t]$, satisfying the conditions of Corollary~\ref{cor:novi}. This finally confirms Conjecture~\ref{conjecture} up to order~$2\,500$.

We also mention that for each $n \ge 3$, a nut graph with $(o_v,o_e) = (2,3)$ of order $n^2$ (see orders $361$, $529$, and $1\,849$ in {Table~\ref{table:others},} or the graph on the right-hand side of Figure~\ref{pic:3n}) can be constructed as a graph from Proposition~\ref{prop:general} by taking $a=\kappa = n - 1$, and $m = 1$,  $t = n - 1$, with $[m] = 1||$, and $[t] = \frac{n - 1}{2}||2$. We point out again that the fact that for each odd~$n\ge3$ a nut graph with $(o_v,o_e)=(2,3)$ of order~$n^2$ exists was shown already in~\cite{BFP24}.

\section{{Computational data and further research}}\label{sec:5}

As already mentioned, Corollaries \ref{cor:3n}, \ref{cor:tetra} and \ref{cor:novi} provide examples of nut graphs with $(o_v, o_e) = (2,3)$ for all odd non-prime orders up to~$2\,500$. The first order that is not taken care of by these corollaries is~$2\,839 = 17\cdot 167$. To see what happens with larger orders, we performed extensive computations for all odd non-prime orders up to~$1\,000\,000$. Some obtained data is gathered in Table~\ref{table:data}, where for each of the chosen~$18$ bounds~$N$, $X$ denotes the set of all odd non-prime orders up to~$N$, $X_1$ denotes the set of orders in~$X$, remaining after applying Corollary~\ref{cor:3n}, $X_2$ denotes the set of orders in~$X_1$, remaining after applying Corollary~\ref{cor:tetra}, and $X_3$ denotes the set of orders in~$X_2$, remaining after applying Corollary~\ref{cor:novi}.

\begin{table}[h!]
\centering
\renewcommand{\arraystretch}{1.0}
\begin{adjustbox}{width=0.77\textwidth}
\begin{tabular}{cccccccc}
\toprule
$N$ & $|X|$ & $|X_1|$ & $|X_1|/|X|$ & $|X_2|$ & $|X_2|/|X|$ & $|X_3|$ & $|X_3|/|X|$ \\
\midrule
$1\,000 $&$ 332 $&$ 9 $&$ 0.02710 $&$ 6 $&$ 0.01807 $&$ 0 $&$ 0$ \\
$2\,500 $&$ 883 $&$ 36 $&$ 0.04077 $&$ 17 $&$ 0.01925 $&$ 0 $&$ 0$ \\
$5\,000 $&$ 1\,831 $&$ 80 $&$ 0.04369 $&$ 39 $&$ 0.02129 $&$ 2 $&$ 0.00109$ \\
$10\,000 $&$ 3\,771 $&$ 172 $&$ 0.04561 $&$ 86 $&$ 0.02280 $&$ 6 $&$ 0.00159$ \\
$20\,000 $&$ 7\,738 $&$ 344 $&$ 0.04445 $&$ 182 $&$ 0.02352 $&$ 10 $&$ 0.00129$ \\
$30\,000 $&$ 11\,755 $&$ 535 $&$ 0.04551 $&$ 290 $&$ 0.02467 $&$ 15 $&$ 0.00127$ \\
$40\,000 $&$ 15\,797 $&$ 722 $&$ 0.04570 $&$ 391 $&$ 0.02475 $&$ 19 $&$ 0.00120$ \\
$50\,000 $&$ 19\,867 $&$ 901 $&$ 0.04535 $&$ 490 $&$ 0.02466 $&$ 22 $&$ 0.00110$ \\
$100\,000 $&$ 40\,408 $&$ 1\,813 $&$ 0.04486 $&$ 991 $&$ 0.02452 $&$ 46 $&$ 0.00113$ \\
$200\,000 $&$ 82\,016 $&$ 3\,561 $&$ 0.04341 $&$ 1\,991 $&$ 0.02427 $&$ 85 $&$ 0.00103$ \\
$300\,000 $&$ 124\,003 $&$ 5\,337 $&$ 0.04303 $&$ 2\,973 $&$ 0.02397 $&$ 132 $&$ 0.00106$ \\
$400\,000 $&$ 166\,140 $&$ 7\,074 $&$ 0.04257 $&$ 3\,934 $&$ 0.02367 $&$ 180 $&$ 0.00108$ \\
$500\,000 $&$ 208\,462 $&$ 8\,745 $&$ 0.04195 $&$ 4\,881 $&$ 0.02341 $&$ 227 $&$ 0.00108$ \\
$600\,000 $&$ 250\,902 $&$ 10\,435 $&$ 0.04158 $&$ 5\,810 $&$ 0.02315 $&$ 267 $&$ 0.00106$ \\
$700\,000 $&$ 293\,457 $&$ 12\,111 $&$ 0.04127 $&$ 6\,760 $&$ 0.02303 $&$ 317 $&$ 0.00108$ \\
$800\,000 $&$ 336\,049 $&$ 13\,834 $&$ 0.04116 $&$ 7\,728 $&$ 0.02299 $&$ 364 $&$ 0.00108$ \\
$900\,000 $&$ 378\,726 $&$ 15\,513 $&$ 0.04096 $&$ 8\,701 $&$ 0.02297 $&$ 417 $&$ 0.00110$ \\
$1\,000\,000 $&$ 421\,502 $&$ 17\,254 $&$ 0.04093 $&$ 9\,718 $&$ 0.02305 $&$ 457 $&$ 0.00108$ \\
\bottomrule
\end{tabular}
\end{adjustbox}
\caption{Computational data up to order a million.}
\label{table:data}
\end{table}

We find that among the $421\,502$ odd non-prime orders up to~$1\,000\,000$, only $457$ of them remain without an example of a nut graph with $(o_v,o_e)=(2,3)$ after applying Corollaries \ref{cor:3n}, \ref{cor:tetra} and \ref{cor:novi}. This is less than $0.11$~percent. In Table~\ref{table:remaining} we state the first~$66$ remaining orders (which are precisely the remaining orders up to~$150\,000$).

\begin{table}[h!]
\centering
\renewcommand{\arraystretch}{1.0}
\begin{adjustbox}{width=1\textwidth}
\begin{tabular}{ccccccccccc}
\toprule
$2\,839$ & $4\,313$ & $5\,377$ & $6\,103$ & $9\,101$ & $9\,557$ & $11\,153$ & $11\,761$ & $12\,631$ & $16\,711$ & $23\,237$ \\
$24\,289$ & $25\,483$ & $27\,319$ & $28\,339$ & $30\,379$ & $34\,459$ & $37\,519$ & $38\,413$ & $42\,617$ & $45\,581$ & $45\,679$ \\ 
$53\,821$ & $57\,103$ & $58\,939$ & $59\,953$ & $59\,959$ & $60\,163$ & $62\,203$ & $62\,809$ & $64\,963$ & $66\,691$ & $69\,521$ \\ 
$69\,689$ & $71\,977$ & $72\,257$ & $77\,299$ & $79\,741$ & $79\,993$ & $81\,377$ & $81\,493$ & $84\,001$ & $86\,683$ & $92\,263$ \\ 
$94\,901$ & $98\,029$ & $105\,181$ & $108\,733$ & $110\,333$ & $112\,183$ & $114\,181$ & $114\,223$ & $114\,817$ & $116\,059$ & $119\,467$ \\ 
$120\,751$ & $127\,891$ & $129\,091$ & $129\,931$ & $133\,249$ & $137\,281$ & $137\,789$ & $139\,651$ & $140\,513$ & $143\,119$ & $145\,217$ \\
\bottomrule
\end{tabular}
\end{adjustbox}
\caption{The remaining orders up to~$150\,000$ after applying Corollaries~\ref{cor:3n}, \ref{cor:tetra} and~\ref{cor:novi}.}
\label{table:remaining}
\end{table}

In addition, Table~\ref{table:data} suggests that Corollary~\ref{cor:3n} alone substantially reduces the number of odd non-prime orders, eliminating over $95$~percent of them. When Corollary~\ref{cor:tetra} and Corollary~\ref{cor:novi} are applied, the proportion of remaining orders seems to be close to $0.11$~percent. Moreover, it seems that this trend goes on also when considering bounds~$N$ exceeding one million.

It is interesting to note (but not very surprising given the nature of Corollaries~\ref{cor:3n}, \ref{cor:tetra} and~\ref{cor:novi}), that each of the $457$ orders up to~$1\,000\,000$, which is not covered by the three corollaries, is a product of two distinct primes. It thus seems that in order to completely resolve Conjecture~\ref{conjecture}, one of the major steps is to consider graphs of order a product of two primes.

We conclude the paper by the following remarks on the possibilities to get new examples of nut graphs with $(o_v,o_e)=(2,3)$. First, note that Proposition~\ref{prop:general} offers more possibilities than what was used in its three corollaries. Namely, it allows any non-singular arc-transitive connected graph to be used as~$\Lambda_1$ from the proof of the proposition. Second, Proposition~\ref{prop:general} could be generalized in the sense that we allow more general types of non-singular arc-transitive graphs for~$\Lambda_4$ and~$\Lambda_5$. Next, Theorem~\ref{theorem:K} offers the possibility that the orbital~$\Delta_2$ is non-diagonal and has nothing to do with~$\Delta_1$, while Construction~\ref{cons:merge} is not restricted to the case that~$\Delta_3=\mathcal{I}$. 
Finally, there also exist nut graphs with $(o_v,o_e)=(2,3)$ that do not arise from Construction~\ref{cons:merge}. In fact, noteworthy examples can be found already by looking at the smallest order not covered by~\cite{BFP24}, namely~$35$. 

In Figures~\ref{pic:Petersen} and~\ref{pic:Heawood}, three such examples are drawn. With reference to notation from Lemma~\ref{lemma:1}, in each of them the vertices and edges of~$\Gamma_1$ are drawn in teal color, the vertices of~$\Gamma_2$ and the edges between $\Gamma_1$ and~$\Gamma_2$ are drawn in black color, while the edges of~$\Gamma_2$ are drawn in gray color.

\begin{figure}[h]
\centering
\begin{tikzpicture}[scale=0.55, rotate=90]
\def\rzero{1.5cm}  
\def\rone{6cm}   
\def\rtwo{8cm}   

\def\n{5}

\def\colorA{gray!20}
\def\colorA{teal!60}
\def\colorB{black!100}
\def\colorC{black!100}
\def\colorD{black!100}
\def\colorE{blue!70}

\foreach \i in {0,...,\numexpr\n-1} {
    \def\angle{360/\n * \i}

    \node[draw, circle, fill=\colorA, minimum size=8pt, inner sep=2pt] (\i) at ({\angle}:\rzero) {};

    \node[draw, circle, fill=\colorB, minimum size=4pt, inner sep=1.5pt] (b\i c) at ({\angle}:\rone) {};
    \node[draw, circle, fill=\colorC, minimum size=4pt, inner sep=1.5pt] (b\i l) at ({\angle - 5}:\rone) {};
    \node[draw, circle, fill=\colorD, minimum size=4pt, inner sep=1.5pt] (b\i r) at ({\angle + 5}:\rone) {};

    \node[draw, circle, fill=\colorB, minimum size=4pt, inner sep=1.5pt] (a\i c) at ({\angle}:\rtwo) {};
    \node[draw, circle, fill=\colorC, minimum size=4pt, inner sep=1.5pt] (a\i l) at ({\angle - 8}:\rtwo) {};
    \node[draw, circle, fill=\colorD, minimum size=4pt, inner sep=1.5pt] (a\i r) at ({\angle + 8}:\rtwo) {};
}

\foreach \i in {0,...,\numexpr \n-1} {
    \foreach \j in {0,...,\numexpr \n-1} {
        \ifnum\i<\j
        \draw[\colorA] (\i) to (\j);
        \fi
    }
}
\foreach \i in {0,...,\numexpr \n-1} {
    \pgfmathsetmacro\ipone{int(mod(\i+1, \n))}   
    \pgfmathsetmacro\iptwo{int(mod(\i+2, \n))}   
    \pgfmathsetmacro\imone{int(mod(\i-1+\n, \n))} 
    \pgfmathsetmacro\imtwo{int(mod(\i-2+\n, \n))} 

    \foreach \j in {c,l,r} {
        \draw[bend right=40] (\i) to (a\ipone \j);
        \draw[bend left=40] (\i) to (a\imone \j);
        \draw[bend right=60] (\i) to (b\iptwo \j);
        \draw[bend left=60] (\i) to (b\imtwo \j);
        }
    \foreach \j in {c,l,r} {
        \foreach \k in {c,l,r} {
            \ifx\j\k
            \else
                \draw[gray!65] (a\i \j) to (a\ipone \k); 
                \draw[gray!65, bend right=10] (b\i \j) to (b\iptwo \k);
                \draw[gray!65] (a\i \j) to (b\i \k);
            \fi
        }
    }
}

\end{tikzpicture}

\caption{An example of a nut graph with parameter $6$-tuple $\langle 5,4,12,30,6,2 \rangle$.}
\label{pic:Petersen} 

\end{figure}
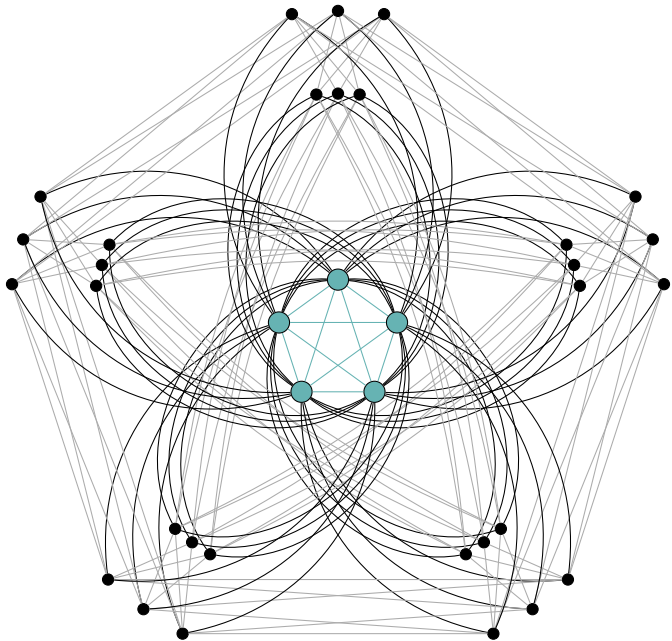

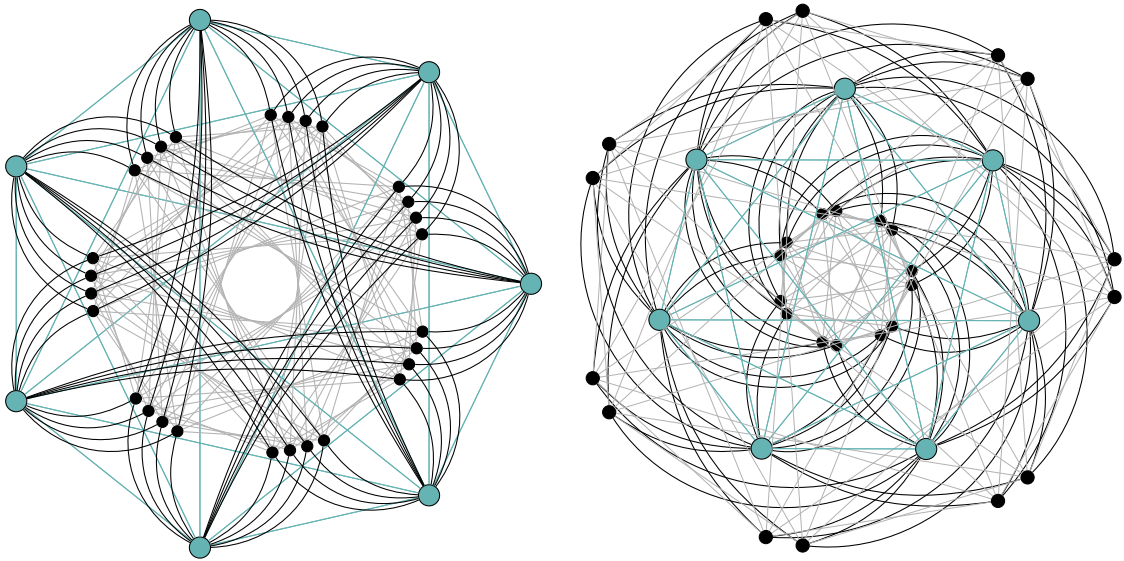
\begin{figure}[h!]
\centering
\begin{tikzpicture}[scale=0.45]
\def\rzero{8cm}  
\def\rone{2cm}   
\def\rtwo{5cm}   

\def\n{7}

\def\colorA{gray!20}
\def\colorA{teal!60}
\def\colorB{black!100}
\def\colorC{black!100}
\def\colorD{black!100}
\def\colorE{black!100}

\foreach \i in {0,...,6} {
    \def\angle{360/\n * \i}

    \node[draw, circle, fill=\colorA, minimum size=8pt, inner sep=2pt] (\i) at ({\angle}:\rzero) {};

    \node[draw, circle, fill=\colorD, minimum size=4pt, inner sep=1.5pt] (a\i) at ({\angle + 17}:\rtwo) {};
    \node[draw, circle, fill=\colorE, minimum size=4pt, inner sep=1.5pt] (b\i) at ({\angle + 23}:\rtwo) {};    
    \node[draw, circle, fill=\colorD, minimum size=4pt, inner sep=1.5pt] (c\i) at ({\angle + 29}:\rtwo) {};
    \node[draw, circle, fill=\colorE, minimum size=4pt, inner sep=1.5pt] (d\i) at ({\angle + 35}:\rtwo) {};
}

\foreach \i in {0,...,6} {    
    \foreach \j in {0,...,6} {
        \ifnum\i<\j   
            \draw[\colorA, bend right=0] (\i) to (\j);  
            \draw[\colorA, bend right=0] (\i) to (\j); 
            \draw[gray!60, bend right=0] (a\i) to (a\j);    
            \draw[gray!60, bend right=0] (b\i) to (b\j);   
            \draw[gray!60, bend right=0] (c\i) to (c\j);  
            \draw[gray!60, bend right=0] (d\i) to (d\j);   
        \fi
    }
}

\foreach \i in {0,...,6} { 
    \pgfmathsetmacro\imone{int(mod(\i+\n-1, \n))}   
    \pgfmathsetmacro\ipfive{int(mod(\i+2, \n))}   

    \draw[bend left=40] (\i) to (a\imone);
    \draw[bend left=35] (\i) to (b\imone);
    \draw[bend left=30] (\i) to (c\imone);
    \draw[bend left=25] (\i) to (d\imone);
    \draw[bend right=25] (\i) to (a\i);
    \draw[bend right=30] (\i) to (b\i);
    \draw[bend right=35] (\i) to (c\i);
    \draw[bend right=40] (\i) to (d\i);
    \draw[bend left=10] (\i) to (a\ipfive);
    \draw[bend left=10] (\i) to (b\ipfive);
    \draw[bend left=10] (\i) to (c\ipfive);
    \draw[bend left=10] (\i) to (d\ipfive);

}

\end{tikzpicture}
\begin{tikzpicture}[scale=0.45]
\def\rzero{5.6cm}  
\def\rone{2cm}   
\def\rtwo{8cm}   

\def\n{7}

\def\colorA{gray!20}
\def\colorA{teal!60}
\def\colorB{black!100}
\def\colorC{black!100}
\def\colorD{black!100}
\def\colorE{black!100}

\foreach \i in {0,...,6} {
    \def\angle{360/\n * \i}

    \node[draw, circle, fill=\colorA, minimum size=8pt, inner sep=2pt] (\i) at ({\angle-13}:\rzero) {}; 

    \node[draw, circle, fill=\colorB, minimum size=4pt, inner sep=1.5pt] (b\i l) at ({\angle - 6}:\rone) {};
    \node[draw, circle, fill=\colorC, minimum size=4pt, inner sep=1.5pt] (b\i r) at ({\angle + 6}:\rone) {};

    \node[draw, circle, fill=\colorD, minimum size=5pt, inner sep=1.5pt] (a\i l) at ({\angle - 4}:\rtwo) {};
    \node[draw, circle, fill=\colorE, minimum size=5pt, inner sep=1.5pt] (a\i r) at ({\angle + 4}:\rtwo) {};
}

\foreach \i in {0,...,6} {
    \pgfmathsetmacro\ipone{int(mod(\i+1, \n))}   
    \pgfmathsetmacro\iptwo{int(mod(\i+2, \n))}   
    \pgfmathsetmacro\imone{int(mod(\i-1+\n, \n))} 
    \pgfmathsetmacro\imtwo{int(mod(\i-2+\n, \n))} 

    \draw[bend right=20] (\i) to (a\ipone l);
    \draw[bend right=22] (\i) to (a\ipone r);
    \draw[bend left=22] (\i) to (a\imone l);
    \draw[bend left=20] (\i) to (a\imone r);
    \draw[bend left=42] (\i) to (a\imtwo l);
    \draw[bend left=50] (\i) to (a\imtwo r);

    \draw[bend right=20] (\i) to (b\ipone l);
    \draw[bend right=25] (\i) to (b\ipone r);
    \draw[bend right=45] (\i) to (b\iptwo l);
    \draw[bend right=60] (\i) to (b\iptwo r);
    \draw[bend left=30] (\i) to (b\imone l);
    \draw[bend left=20] (\i) to (b\imone r);
}

\foreach \i in {0,...,6} {
    \foreach \j in {0,...,6} {
        \ifnum\i<\j
            \draw[gray!60] (b\i l) -- (b\j l);
            \draw[gray!60] (b\i r) -- (b\j r);
            \draw[gray!60, line width=0.2pt] (a\i l) -- (a\j l);
            \draw[gray!60, line width=0.2pt] (a\i r) -- (a\j r);
            \draw[\colorA, line width=0.2pt] (\i) -- (\j);
            \draw[\colorA, line width=0.2pt] (\i) -- (\j);
        \fi
    }
}

\end{tikzpicture}
\caption{Two examples of nut graphs with parameter $6$-tuple $\langle 7,6,12,28,6,3 \rangle$.}
\label{pic:Heawood} 

\end{figure}

For the graph~$\Gamma$ on Figure~\ref{pic:Petersen}, $\Gamma_1 \cong K_5$ and $\Gamma_2 \cong \mathrm{GP}(5,2) \times K_3$, where $\mathrm{GP}(5,2)$ is the Petersen graph. Using a computer, one can check that~$\Gamma$ is a nut graph with $(o_v,o_e)=(2,3)$ and parameter $6$-tuple $\langle 5,4,12,30,6,2 \rangle$. We claim that this graph does not arise from Construction~\ref{cons:merge}. Suppose to the contrary, that there exist integers $n,m,t$, groups $G,H,K$ and orbitals $\Delta_i$, $i \in \{1,2,3,4,5\}$ as in Construction~\ref{cons:merge}, such that $\Gamma = \GG(\Delta_1,\Delta_2,\Delta_3,\Delta_4,\Delta_5)$. Since $n_1=5$, it clearly follows that~$n=5$, $m=1$ and $t=6$. Next, since $k_1=4$, $\mathcal{O}(\Delta_1)=K_5$, and so $G$ has only two orbitals in its action on~$\ZZ_5$, namely~$\Delta_1$ and~$\mathcal{I}$. Therefore, $12 = d_1 = |\Delta_3(0)|\cdot t \in \{1\cdot 6, 4\cdot 6\}$, a contradiction. Hence, $\Gamma$ does not arise from Construction~\ref{cons:merge}, as claimed.
For the two graphs on Figure~$\ref{pic:Heawood}$, $\Gamma_1 \cong K_7$ and $\Gamma_2 \cong 4K_7$. A computer check reveals that both graphs are nut graphs with $(o_v,o_e)=(2,3)$ and parameter $6$-tuple $\langle 7,6,12,28,6,3 \rangle$. That none of them arises from Construction~\ref{cons:merge}, can be shown in a similar way as for the graph from Figure~\ref{pic:Petersen}. 

Finally, the automorphism groups of the three graphs in Figures~\ref{pic:Petersen} and~\ref{pic:Heawood} are also interesting. The automorphism group of the graph on Figure~\ref{pic:Petersen} is isomorphic to the semidirect product $A_5 \rtimes D_6 = \mathrm{PSL}(2,5) \rtimes D_6$. Moreover, the automorphism group of the graph on the left-hand side of Figure~\ref{pic:Heawood} is isomorphic to the direct product $\mathrm{PSL}(2,7) \times S_4$. Note that contracting each group of $4$ black vertices to a single vertex and deleting all gray and teal edges results in the well-known Heawood graph. The presence of the simple group $\mathrm{PSL}(2,7)$ in the automorphism group is thus not that surprising. In contrast, the automorphism group of the graph on the right-hand side of Figure~\ref{pic:Heawood} is solvable. In fact, it is isomorphic to a semidirect product $(\ZZ_{2} \times \ZZ_{14}) \rtimes \ZZ_6$. But what is fascinating about it, is that while the automorphism group of each of the first two examples acts arc-transitively on each of the subgraphs~$\GG_1$ and~$\GG_2$, the automorphism group of the last one acts half-arc-transitively on $\GG_2$ (that is, it acts transitively on the set of its vertices and edges, but not on the set of its arcs).

This shows that there are still many fascinating aspects of nut graphs with a high degree of symmetry, and in particular the ones corresponding to Conjecture~\ref{conjecture}, that could be investigated in the future.

\section*{Acknowledgments and Declarations}
{\bf Funding:} \\
The authors acknowledge the financial support by the Slovenian Research and Innovation Agency (Young researchers program, research program P1-0285, and research project J1-50000).
\medskip

\noindent
{\bf Competing interests:} \\
The authors declare that they have no relevant competing financial interests.

\printbibliography

\end{document}